\documentclass[a4paper]{amsart}

\usepackage{graphicx} %
\usepackage{amsmath,amssymb,amsthm,thmtools}
\usepackage{mathtools}
\usepackage{bbm}
\usepackage[shortlabels]{enumitem}
\usepackage{ifthen}
\usepackage{xcolor}
\usepackage{tikz}
\usepackage{pgfplots}
\pgfplotsset{compat=newest}
\usepgfplotslibrary{groupplots}
\usepackage{xstring}
\usetikzlibrary{calc, external}
\tikzset{external/up to date check=md5}
\usepackage{standalone}
\usepackage{float}
\usepackage{subcaption}
\usepackage[ruled,vlined]{algorithm2e}
\usepackage{booktabs}

\usepackage{hyperref}
\hypersetup{
	colorlinks,
	linkcolor={red!50!black},
	citecolor={blue!50!black},
	urlcolor={blue!50!black}
}

\ifpdf
\hypersetup{
  pdftitle={Domain splitting with localized implicit prediction},
  pdfauthor={Tim Buchholz and Roland Maier}
}
\fi

\newcommand{\nobreakbeforeinlinemath}{%
  \ifhmode
    \ifdim\lastskip>0pt
      \unskip
      \nobreakspace
    \fi
  \fi
}

\AddToHook{cmd/(/before}{\nobreakbeforeinlinemath}

\newcommand{\half}{1/2}
\newcommand{\restr}[2]{{%
\left.\kern-\nulldelimiterspace %
#1 %
\vphantom{\big|} %
\right|_{#2} %
}}
\DeclareMathOperator{\supp}{supp}
\DeclareMathOperator{\dom}{dom}
\DeclareMathOperator{\interior}{int}
\DeclareMathOperator{\diam}{diam}

\DeclarePairedDelimiterX\setc[2]{\{}{\}}{\,#1 \;\delimsize\vert\; #2\,}
\DeclarePairedDelimiter\monosetc{\{}{\}}

\newcommand{\mfrac}[2]{\scalebox{0.8}{\(\dfrac{#1}{#2}\)}{}}
\newcommand{\pmat}[1]{\begin{pmatrix} #1 \end{pmatrix}}
\newcommand{\pmatT}[1]{(#1)^\top} %
\newcommand{\grad}{\nabla}
\renewcommand{\~}[1]{\widetilde{#1}}
\newcommand{\dint}[1]{\, \mathrm{d} #1}
\renewcommand{\d}[1]{\partial_{#1}}
\DeclareMathOperator{\Div}{div}
\newcommand{\Th}[1][]{
	\ifthenelse{ \equal{#1}{} }
	{\mathcal{T}_h}
	{\mathcal{T}_h(#1)}
} 
\newcommand{\card}{\mathop{\mathrm{card}}}

\newcommand{\PkK}{\mathbb{P}_k(K)}
\newcommand{\PkTh}[1][]{\mathcal{P}_k(\Th[#1])}

\newcommand{\PkZeroTh}[1][]{\mathcal{P}_{k,0}(\Th[#1])}

\DeclarePairedDelimiterX{\innerprod}[1]{(}{)}{#1}
\NewDocumentCommand{\ip}{s o m}{%
  \IfBooleanTF{#1}
    {%
      \IfNoValueTF{#2}
        {\innerprod*{#3}}
        {\innerprod*{#3}_{#2}}%
    }
    {%
      \IfNoValueTF{#2}
        {\innerprod[\big]{#3}}
        {\innerprod[\big]{#3}_{#2}}%
    }%
}

\newcommand{\abilSymb}{a}
\newcommand{\bbilSymb}{b}

\DeclarePairedDelimiterX{\formargs}[1]{(}{)}{#1}
\NewDocumentCommand{\abil}{s o m}{%
  \IfBooleanTF{#1}
    {%
      \IfNoValueTF{#2}
        {\abilSymb\formargs*{#3}}
        {\abilSymb_{#2}\formargs*{#3}}%
    }
    {%
      \IfNoValueTF{#2}
        {\abilSymb\formargs[\big]{#3}}
        {\abilSymb_{#2}\formargs[\big]{#3}}%
    }%
}
\NewDocumentCommand{\bbil}{s o m}{%
  \IfBooleanTF{#1}
    {%
      \IfNoValueTF{#2}
        {\bbilSymb\formargs*{#3}}
        {\bbilSymb_{#2}\formargs*{#3}}%
    }
    {%
      \IfNoValueTF{#2}
        {\bbilSymb\formargs[\big]{#3}}
        {\bbilSymb_{#2}\formargs[\big]{#3}}%
    }%
}

\DeclarePairedDelimiter{\plainabs}{\lvert}{\rvert}
\DeclarePairedDelimiter{\plainnorm}{\lVert}{\rVert}
\NewDocumentCommand{\abs}{s o O{} m}{%
  \IfBooleanTF{#1}
    {%
      \IfNoValueTF{#2}
        {\plainabs*{#4}}
        {\plainabs*{#4}_{#2}}%
    }
    {%
      \IfNoValueTF{#2}
        {\plainabs[#3]{#4}}
        {\plainabs[#3]{#4}_{#2}}%
    }%
}
\NewDocumentCommand{\norm}{s o O{} m}{%
  \IfBooleanTF{#1}
    {%
      \IfNoValueTF{#2}
        {\plainnorm*{#4}}
        {\plainnorm*{#4}_{#2}}%
    }
    {%
      \IfNoValueTF{#2}
        {\plainnorm[#3]{#4}}
        {\plainnorm[#3]{#4}_{#2}}%
    }%
}
\NewDocumentCommand{\trnormmanual}{m m}{%
  \mathopen{}#1|\!#1|\!#1|#2#1|\!#1|\!#1|\mathclose{}%
}
\NewDocumentCommand{\tnorm}{s o O{} m}{%
  \IfBooleanTF{#1}
    {%
      \IfNoValueTF{#2}
        {\left|\!\left|\!\left|#4\right|\!\right|\!\right|}
        {\left|\!\left|\!\left|#4\right|\!\right|\!\right|_{#2}}%
    }
    {%
      \IfNoValueTF{#2}
        {\trnormmanual{#3}{#4}}
        {\trnormmanual{#3}{#4}_{#2}}%
    }%
}

\newcommand{\fulldomain}{\Omega}
\newcommand{\dfulldomain}{\partial \fulldomain}
\newcommand{\fulldomainc}{\overline{\fulldomain}}

\newcommand{\ov}{\delta}%
\newcommand{\ovp}{\ell_\mathrm{o}} %
\newcommand{\povp}[1][]{%
  \nobreakbeforeinlinemath
	\ifthenelse{ \equal{#1}{} }
	{\ensuremath{\ell_\mathrm{p}}}
	{\ensuremath{\ell_{\mathrm{p},#1}}}
}
\newcommand{\numberSD}{N}
\newcommand{\SD}[1]{\fulldomain_{#1}}

\newcommand{\SDc}[1]{\overline{\fulldomain}_{#1}}

\newcommand{\SDov}[1]{\fulldomain_{#1}^{\ov}}
\newcommand{\dSDov}[1]{\partial \SDov{#1}}
\newcommand{\SDovc}[1]{\overline{\fulldomain}_{#1}^{\ov}}

\newcommand{\sumSD}{\sum_{i=1}^{\numberSD}}
\newcommand{\SDovint}[1]{\Gamma_{#1}^{\ov}}

\newcommand{\LiftStrip}[1]{\mathcal{S}(\SDovint{#1})}

\newcommand{\averaging}{\zeta}
\newcommand{\blifting}[1]{S_{#1}}

\newcommand{\ts}[1][]{
  \nobreakbeforeinlinemath
	\ifthenelse{ \equal{#1}{} }
	{\ensuremath{\tau}}
	{\ensuremath{\frac{\tau}{#1}}}
}
\newcommand{\tss}[1][]{
  \nobreakbeforeinlinemath
	\ifthenelse{ \equal{#1}{} }
	{\ensuremath{\tau^2}}
	{\ensuremath{\frac{\tau^2}{#1}}}
}
\newcommand{\numberTS}{n_T}

\newcommand{\Patch}[2]{P_{#1}(#2)}
\newcommand{\bLayer}[1]{\Sigma_{#1}}
\newcommand{\NodalInt}{\mathcal{I}_h}
\newcommand{\id}{\mathcal{I}}
\newcommand{\cutoff}[1]{\eta_{#1}}

\newcommand{\LTproj}{\Pi_h}
\newcommand{\Rproj}{\mathcal{R}_h}

\newcommand{\solutionVector}{x}
\newcommand{\uComponent}{u}
\newcommand{\vComponent}{v}

\newcommand{\uExakt}[1][]{\uComponent^{#1}}
\newcommand{\vExakt}[1][]{\vComponent^{#1}}

\newcommand{\solDS}[1]{\solutionVector_{\normalfont\text{DS}}^{#1}} 
\newcommand{\uDS}[1]{\uComponent_{\normalfont\text{DS}}^{#1}}
\newcommand{\vDS}[1]{\vComponent_{\normalfont\text{DS}}^{#1}}

\newcommand{\solDSloc}[2]{\solutionVector^{#2}_{#1}}
\newcommand{\uDSloc}[2]{\uComponent_{#1}^{#2}}
\newcommand{\vDSloc}[2]{\vComponent_{#1}^{#2}}

\newcommand{\solCNloc}[1]{\widehat{\solutionVector}_{\normalfont\text{CN}}^{#1}}
\newcommand{\uCNloc}[1]{\widehat{\uComponent}_{\normalfont\text{CN}}^{#1}}
\newcommand{\vCNloc}[1]{\widehat{\vComponent}_{\normalfont\text{CN}}^{#1}}

\newcommand{\solCN}[1]{\solutionVector_{\normalfont\text{CN}}^{#1}}
\newcommand{\uCN}[1]{\uComponent_{\normalfont\text{CN}}^{#1}}
\newcommand{\vCN}[1]{\vComponent_{\normalfont\text{CN}}^{#1}}

\newcommand{\solCNmod}[1]{\widetilde{\solutionVector}_{\normalfont\text{CN}}^{\,#1}}
\newcommand{\uCNmod}[1]{\widetilde{\uComponent}_{\normalfont\text{CN}}^{\,#1}}
\newcommand{\vCNmod}[1]{\widetilde{\vComponent}_{\normalfont\text{CN}}^{\,#1}}

\newcommand{\solLP}[2]{\~{\solutionVector}^{#2}_{\star,#1}}
\newcommand{\uLP}[2]{\widetilde{\uComponent}_{\star,#1}^{#2}}
\newcommand{\vLP}[2]{\widetilde{\vComponent}_{\star,#1}^{#2}}

\newcommand{\Cdata}{C_{\mathrm{data}}}
\newcommand{\Cblift}{c_{S}}
\newcommand{\cavg}{c_{\mathrm{avg}}}
\newcommand{\CLoc}{C_{\tau/h}} %
\newcommand{\Cint}{c_{\mathcal{I}}}
\newcommand{\Cinv}{c_{\mathrm{inv}}}
\newcommand{\CCutoff}{c_{\eta}}

\newtheorem{theorem}{Theorem}[section]
\newtheorem{lemma}[theorem]{Lemma}
\newtheorem*{remark}{Remark}

\newtheorem{definition}[theorem]{Definition}
\newtheorem{assumption}[theorem]{Assumption}

\makeatletter
\renewcommand{\paragraph}{\@startsection{paragraph}{4}{\z@}%
  {1.0ex plus .2ex minus .1ex}%
  {-1em}%
  {\normalfont\normalsize\bfseries}}
\makeatother

\begin{document}
\title[Domain splitting with localized implicit prediction]{Domain splitting with localized implicit prediction for wave equations}
\author{Tim Buchholz}
\author{Roland Maier}
\address{T. Buchholz, R. Maier,\hfill\break
Institute for Applied and Numerical Mathematics,\hfill\break
Karlsruhe Institute of Technology (KIT), \hfill\break
D-76128 Karlsruhe, Germany}
\email{\{tim.buchholz,roland.maier\}@kit.edu}


\begin{abstract} 
    We propose and analyze a non-iterative domain decomposition time integrator for the linear acoustic wave equation.
    The method is based on an overlapping decomposition in space and uses a localized implicit prediction to provide artificial boundary data at the interfaces.
    We use conforming finite elements of arbitrary order to discretize in space and allow spatially varying material coefficients.
    We prove an error bound relative to the global Crank--Nicolson approximation that decays exponentially with the size of both the subdomain and prediction overlaps.
    Compared to earlier work, this new method avoids a CFL-type step size restriction and allows the use of finite elements of arbitrary order.
    The analysis combines a discrete Saint-Venant principle with new decay and truncation estimates for the localized prediction.
    Numerical experiments support the analysis and demonstrate stable behavior for all tested time step sizes, accuracy comparable to that of the global Crank--Nicolson approximation, and the potential for efficient parallel execution.
\end{abstract}

\maketitle

\section{Introduction}
\label{dslp:sec:introduction}
We consider the first-order formulation of the linear acoustic wave equation given by 
\begin{equation}
    \label{dslp:eq:waveEqFirstOrderShort}
    \partial_t u = v,
    \qquad
    \partial_t v = B^{-1}\Div(A\grad u) + f
    \quad \text{in } \fulldomain\times(0,T]
\end{equation}
on a bounded polygonal or polyhedral domain \(\fulldomain\subset\mathbb R^d\), \(d\in\{1,2,3\}\), and for a final time \(T>0\).
The material coefficients \(A,B\in L^\infty(\fulldomain)\) are assumed to satisfy
\[
    \alpha \leq A(x) \leq \beta,
    \qquad
    \alpha \leq B(x)\leq \beta
\]
for almost every \(x\in\fulldomain\), with constants \(0<\alpha\leq\beta\)\footnote{Note that the common bounds on \(A\) and \(B\) only simplify the presentation and are not a restriction.}. 
Here, the coefficient \(A\) models the inverse density of the material, while \(B\) is a positive compressibility coefficient. 
The generally space-dependent wave speed is determined by the square root of the ratio \(B^{-1}A\).
The system~\eqref{dslp:eq:waveEqFirstOrderShort} is complemented by the initial conditions~\(u(\cdot,0)=u^0\) and~\(v(\cdot,0)=v^0\) in \(\fulldomain\), and homogeneous Dirichlet boundary conditions~\(u(\cdot,t)|_{\partial\fulldomain}=0\) for \(t\in(0,T]\), which simplify the presentation but impose no real restriction.

After spatial discretization, the choice of a time integrator involves a fundamental tradeoff.
Explicit methods such as the \emph{leapfrog} scheme \cite{Jol03,Chr09} are computationally inexpensive per time step, but their stability requires a CFL-type condition that couples the time step size to the spatial mesh size.
Implicit methods such as the \emph{Crank--Nicolson} scheme \cite{CraN47,Kar11} avoid this restriction, but require the solution of a globally coupled linear system at every time step \cite{DoeHKRSW23}.

A natural approach is therefore to decompose the spatial domain, replace the global implicit problem by several smaller subdomain problems, and combine their solutions to a global approximation.
This is the basic idea of \emph{domain decomposition} (DD) methods.
Many established DD methods are iterative: they repeatedly solve the subdomain problems while updating the coupling conditions at their artificial interfaces; see, e.g., \cite{TosW05,Gan08,GanZ22}.
Here, we instead focus on non-iterative DD methods, which require no additional fixed-point or interface iteration within a time step.

Non-iterative DD time integrators include several distinct approaches.
\emph{Domain decomposition operator splittings}, also known as regionally additive schemes, were introduced in \cite{Vab89} and subsequently analyzed mainly for dissipative and parabolic problems \cite{HanE17,EisH18}.
For hyperbolic systems, \emph{tent pitching} methods provide a local space-time construction, with both mapped \cite{GopMS15,GopSW17,DraGSW22} and unmapped variants \cite{CiaGM24,BonCM26}.
Most closely related to the present work is the \emph{domain splitting method}, introduced for parabolic equations in \cite{BluLR92}.
It combines an overlapping spatial decomposition with temporal extrapolation to predict the artificial interface data and provides the starting point for our method.
A related method was developed in \cite{DawD92} and extended to hyperbolic problems in \cite{DawD94}.

In \cite{BucH25CG}, a domain splitting method for the wave equation was constructed, based on the method of~\cite{BluLR92}, but using a leapfrog-based prediction of the artificial interface data.
Second-order convergence relative to the global Crank--Nicolson approximation was proved under a CFL-type condition in which the admissible time step size scales linearly with the width of the subdomain overlap.
The explicit prediction also restricted the conforming spatial discretization because it required mass lumping.
Consequently, only linear finite elements were considered, as higher-order mass-lumped methods on simplicial meshes would generally require enriched finite element spaces, see, e.g., \cite{GeeMV18}.

In the present work, we replace the explicit predictor with a \emph{localized implicit prediction} inspired by \cite{GalM23}.
The key difference from the approach in \cite{GalM23} lies in the localization mechanism.
While that approach uses a partition of unity to localize the data entering the prediction, we apply a cutoff to the test functions in the variational formulation of the prediction problem.
The resulting \emph{domain splitting scheme with localized implicit prediction} is a non-iterative domain decomposition time integrator.
Both the localized prediction solves and the implicit solves on the overlapping subdomains can be carried out independently for each subdomain and thus in parallel.
Compared with the explicit prediction scheme in \cite{BucH25CG}, the analysis no longer requires a CFL-type step size restriction. 
Instead, we prove that the error relative to the global Crank--Nicolson approximation decays exponentially in both the subdomain and prediction overlap widths, yielding a prescribed algebraic convergence rate under a suitable overlap condition.
Moreover, the new scheme supports standard conforming finite elements of any fixed order \(k\geq1\) without mass lumping, and the analysis allows for spatially varying material coefficients \(A,B\in L^\infty(\fulldomain)\).

The theoretical results combine ideas from \cite{BucH25CG,BucD26,GalM23}, but the localized implicit prediction and the absence of a CFL-type restriction make a novel error analysis necessary.
The proof has three principal ingredients.
First, we adapt the error recursion and averaging framework from \cite{BucH25CG}.
Second, the discrete Saint-Venant principle from \cite{BucD26} controls how errors introduced at the artificial interfaces decay across the subdomain overlap.
Third, inspired by \cite{GalM23}, we develop new decay and truncation estimates for the prediction error generated by the modified localization mechanism.

Decay estimates and localization arguments form the heart of our analysis.
Such arguments go back to the Saint-Venant principle in linear elasticity \cite{StV1856} and later developments for elliptic boundary-value problems \cite{HorK83}.
Related exponential decay and localization phenomena also arise in multiscale methods such as \cite{MalP14,MalP21,MaSD22,AlbMS25,Owh17,OwhS19},
and are used to show the energy decay of the Green's function for linear Schr\"odinger operators \cite{AltHP20}.

The remainder of the paper is organized as follows.
Section~\ref{dslp:sec:preliminaries} introduces the notation and the global discretizations in space and time, while Section~\ref{dslp:sec:method} constructs the domain splitting method with particular emphasis on the localized implicit prediction.
Section~\ref{dslp:sec:main-results} states the convergence result relative to the global Crank--Nicolson approximation and summarizes its three main ingredients.
Sections~\ref{dslp:sec:localizedPrediction}--\ref{dslp:sec:averaging} establish the required ingredients, namely the localized prediction estimate, the discrete Saint-Venant principle, and the averaging stability.
Finally, Section~\ref{dslp:sec:discussion} discusses the overlap condition, while Section~\ref{dslp:sec:numericalExperiments} presents one- and two-dimensional numerical experiments supporting the theory and illustrating the potential for efficient parallel implementation.

\section{Preliminaries}
\label{dslp:sec:preliminaries}
In this section, we introduce the analytical and discrete framework for this paper. 
We first derive a variational formulation of the model problem~\eqref{dslp:eq:waveEqFirstOrderShort} and then discretize it in space using conforming Lagrange finite elements of arbitrary fixed order \(k\geq 1\).
Next, we introduce the required mesh-layer geometry for the subsequent localization arguments. 
Finally, we present the global Crank--Nicolson scheme, which serves as a reference method in our analysis, and establish a stability bound in the discrete energy norm.

\subsection{Spaces, norms, and bilinear forms}

We denote the Lebesgue space of square integrable functions by \(L^2(\fulldomain)\) and write \(\ip*[\fulldomain]{\cdot,\cdot}\) and \(\norm[\fulldomain]{\cdot}\) for its inner product and norm. We mostly work with a version of the inner product, which is weighted by the material coefficient \(B\), given by 
\begin{equation}
    \bbil[\fulldomain]{u,w}
     \coloneq 
     \int_{\fulldomain} B u w \dint{x}, 
\end{equation} 
and denote the associated norm by \(\norm[b,\fulldomain]{\cdot}\).
By \(H^1(\fulldomain)\), we denote the Sobolev space of functions with weak derivatives up to order one in \(L^2(\fulldomain)\) and write \(H_0^1(\fulldomain)\) for the closure of the compactly supported \(C_c^\infty(\fulldomain)\) functions in \(H^1(\fulldomain)\). 
As usual, we denote the respective norm and semi-norm by~\(\norm[H^1(\fulldomain)]{\cdot}\) and \(\abs{\cdot}_{H^1(\fulldomain)}\). 
Motivated by our model problem \eqref{dslp:eq:waveEqFirstOrderShort}, we also introduce the bilinear form
\begin{equation}
    \abil[\fulldomain]{u,w} \coloneq \int_{\fulldomain} A \grad u \cdot \grad w \dint{x},
\end{equation}
and its associated norm \(\norm[a,\fulldomain]{\cdot}\). 
With these definitions, and for sufficiently regular solutions with~\(u(t)\in H_0^1(\fulldomain)\) and \(v(t)\in L^2(\fulldomain)\) for almost every \(t\in(0,T)\), the variational formulation of \eqref{dslp:eq:waveEqFirstOrderShort} reads
\begin{equation}
    \label{dslp:eq:variationalFormulation-continuous}
    \begin{aligned}
         \bbil[\fulldomain]{ \d{t} u, \varphi} &=  \bbil[\fulldomain]{v, \varphi} &&\text{for all } \varphi \in H_0^1(\fulldomain),  \\
         \bbil[\fulldomain]{\d{t} v , \psi} &= -\abil[\fulldomain]{u, \psi} +  \bbil[\fulldomain]{f, \psi}  &&\text{for all } \psi \in H_0^1(\fulldomain)
    \end{aligned}
\end{equation}
almost everywhere in time.
Throughout this work, we will gather the components in a solution vector~\(\pmatT{u \; v}\), whenever possible. 

\subsection{Conforming finite elements}
To discretize the model problem in space, we introduce conforming Lagrange finite elements together with the auxiliary operators and estimates needed in the subsequent analysis. 
Let \(\Th = \Th[\fulldomain]\) be an affine, shape-regular, matching simplicial mesh of \(\fulldomain\), see, e.g., \cite[Definition 8.11]{ErnG21I}, where \(h\) denotes a spatial parameter tending to zero for finer meshes,
typically the largest diameter of all elements.
Throughout, we also assume quasi-uniformity of the mesh\footnote{Assuming quasi-uniformity is not a restriction and only simplifies the presentation to avoid an additional parameter \(h_{\min} = \min_{K\in \Th} \diam K\). In a non quasi-uniform setting, bounds using \(1/h\) must use~\(1/h_{\min}\) instead.}. That is, \(h\) is up to a constant also a lower bound on the diameters of the elements. 
Let~\(\PkK\) denote the set of polynomials in \(d\) variables of degree \(\leq k\) on an element~\(K\in \Th\) for some~\(k\geq 1\). Based on this, we define the finite-dimensional approximation spaces
\begin{equation}
    \begin{aligned}
        \PkTh &\coloneq \setc{v_h \in L^2(\fulldomain)}{\restr{v_h}{K} \in \PkK\;\text{for all } K \in \Th} \cap C^{0}(\fulldomain), \\
        \PkZeroTh &\coloneq  \PkTh \cap H_0^1(\fulldomain).
    \end{aligned}
\end{equation}
For each \(K\in \Th\), we introduce the nodal degrees of freedom
\begin{equation*}
    \mathcal{N}_K \coloneq \monosetc*{\sigma_{K,1},\dots,\sigma_{K,N_k^d}},
\end{equation*}
consisting of the point evaluations \(\sigma_{K,i}(u) = u(a_{K,i})\) at suitable nodal points \(a_{K,i} \in K\) for indices \(i=1,\dots,N_k^d\), as introduced, e.g., in \cite[Section 7.4]{ErnG21I}. We follow the usual convention that mesh cells are closed sets. 
Hence, \((K,\PkK, \mathcal{N}_K)\) is the standard Lagrange finite element of degree \(k\).

A standard ingredient for the error analysis of finite element methods is the inverse inequality. 
\begin{lemma}[Inverse inequality]
	\label{dslp:Lem:InverseEstimate}
    There exists a constant \(\Cinv>0\), depending only on the polynomial degree \(k\), the spatial dimension, and the shape regularity of \(\Th\), such that
    \begin{equation*}
        \abs{\phi_h}_{H^1(K)}^2
        \leq
        \Cinv h^{-2}\norm[K]{\phi_h}^2
    \end{equation*}
    for all \(K\in\Th\) and all \(\phi_h\in\PkTh\).
\end{lemma}
\begin{proof}
	The proof of this standard inverse inequality can be found, e.g., in \cite[Lemma 12.1, Example 12.3]{ErnG21I} or \cite[Lemma 4.5.3]{BreS08}.
\end{proof}

For a function \(u \in C^0(K)\), we define a local interpolant by 
\begin{equation*}
    \mathcal{I}_K u 
    \coloneq \sum_{i=1}^{N_k^d} \sigma_{K,i}(u) \varphi_{a_{K,i}} 
    = \sum_{i=1}^{N_k^d} u(a_{K,i}) \varphi_{a_{K,i}},
\end{equation*}
where \(\varphi_{a_{K,i}}\) denotes the local shape function associated with the nodal point \(a_{K,i}\).
With these local interpolants, the standard Lagrange nodal interpolation \(\NodalInt u \in \PkTh[]\) of a function~\(u\in C^0(\fulldomainc)\) is defined by 
\begin{equation}
    \label{dslp:eq:DefNodalInt}
    \restr{\bigl(\NodalInt u\bigr)}{K} \coloneq \mathcal{I}_K u, \quad \text{for all }K \in \Th .
\end{equation}
Moreover, we define the weighted \(L^2\)-projection \(\LTproj\colon L^2(\fulldomain) \to \PkTh[]\) for \(u \in L^2(\fulldomain)\) by 
\begin{equation}
    \label{dslp:eq:L2Projection}
     \bbil[\fulldomain]{\LTproj u, w_h} =  \bbil[\fulldomain]{u, w_h} \quad \text{for all } w_h \in \PkTh[].
\end{equation}
Similarly, for \(u \in H^1_0(\fulldomain)\), we define the Ritz projection \(\Rproj u \in \PkZeroTh\) based on the bilinear form~\(\abil*[\fulldomain]{\cdot,\cdot}\) by
\begin{equation}
    \label{dslp:eq:RitzProjection}
    \abil[\fulldomain]{\Rproj u, w_h} = \abil[\fulldomain]{u, w_h} \quad \text{for all } w_h \in \PkZeroTh[].
\end{equation}

We relate the two bilinear forms \(\abil[\fulldomain]{\cdot,\cdot}\) and \(\bbil[\fulldomain]{\cdot,\cdot}\) by the linear operator~\(L_h \colon \PkZeroTh[] \to \PkZeroTh[]\), defined for \(u_h \in \PkZeroTh[]\) by
\begin{equation}
    \label{dslp:eq:DefDiscreteOperatorLh}
     \bbil[\fulldomain]{L_h u_h, w_h} = \abil[\fulldomain]{u_h, w_h} \quad \text{for all } w_h \in \PkZeroTh[].
\end{equation}

\subsection{Discrete variational formulation}
Applying the finite element setting from above to \eqref{dslp:eq:variationalFormulation-continuous} leads to a spatially discrete variational formulation posed in the discrete energy space
\begin{equation}
    X_h = X_h(\fulldomain)
    \coloneq
    \bigl(\PkZeroTh, \norm[a]{\cdot}\bigr)
    \times
    \bigl(\PkTh, \norm[b]{\cdot}\bigr),
\end{equation}
with the associated discrete energy norm  
\begin{equation}
    \norm[X_h]{x_h} = \norm[X_h(\fulldomain)]{x_h} \coloneq \Bigl(\norm[a]{u_h}^2 + \norm[b]{v_h}^2\Bigr)^{\half},
\end{equation}
for a solution vector \(x_h = \pmatT{u_h \; v_h}\). 
Throughout the paper, omitted domain indices for norms always refer to the full domain \(\fulldomain\). For instance, we write  
\(\norm[a]{\cdot}=\norm[a,\fulldomain]{\cdot}\) and
\(\norm[b]{\cdot}=\norm[b,\fulldomain]{\cdot}\).

To quantify the size of the data, we work with a constant~\(\Cdata = \Cdata(\uExakt[0],\vExakt[0],f,T)\) given in the definition below. 
\begin{definition}[Well-prepared data]
    \label{dslp:def:data-well-prepared}
    We call the initial data \(u^0 \in H_0^1(\fulldomain)\), \(v^0 \in L^2(\fulldomain)\), and the right-hand side~\(f\in C([0,T]; L^2(\fulldomain))\) \emph{well-prepared} if 
    \begin{equation}
        \label{dslp:eq:def-Cdata}
        \Cdata \coloneq \norm[a]{\uExakt[0]} + \norm[b]{\vExakt[0]} + T \max_{j=0,\dots,\numberTS} \norm[b]{f(t_j)}
    \end{equation}
    is finite and independent of \(h\) and \(\tau\). 
\end{definition}
Let \(\uExakt[0], \vExakt[0]\), and \(f\) be well-prepared according to Definition~\ref{dslp:def:data-well-prepared}. 
We project the initial data to~\(X_h\) using the projections \(\LTproj\) and \(\Rproj\) from \eqref{dslp:eq:L2Projection} and \eqref{dslp:eq:RitzProjection}, respectively, and define the discrete forcing by
\begin{equation}
    \uCN{0} = \Rproj \uExakt[0],
    \qquad
    \vCN{0} = \LTproj \vExakt[0],
    \qquad
    f_h(t) = \LTproj f(t)
    \quad\text{for }t\in[0,T].
\end{equation}
For \(j=0,\dots,\numberTS\), we abbreviate \(f_h^j=f_h(t_j)\). 

In the discrete energy space \(X_h(\fulldomain)\), the semi-discrete counterpart of \eqref{dslp:eq:variationalFormulation-continuous} reads 
\begin{equation}
    \label{dslp:eq:variationalFormulation-discrete}
    \partial_t \pmat{u_h \\ v_h} = \pmat{0 & \id \\ -L_h &0} \pmat{u_h \\ v_h} + \pmat{0 \\ f_h(t)},
\end{equation}
where \(\id\) denotes the identity operator.

\subsection{Patches and mesh layers}
The localization arguments in the later error analysis require local norms and a notion of distance expressed in layers of mesh elements.
Thus, we also use the bilinear forms and norms from above on subdomains~\(D\subset \fulldomain\), which will always match the underlying mesh \(\Th\). That is, there exists a subset \(D_h \subset \Th\) with
\begin{equation}
    \label{dslp:eq:D-matching-Th}
    D = \dom D_h \coloneq \interior\bigl( \cup_{K \in D_h} K \bigr).
\end{equation}
For example, \(X_h(D)\) denotes the discrete energy space on \(D\), and we write 
\begin{equation*}
    \norm[X_h(D)]{x_h} = \Bigl( \norm[a,D]{u_h}^2 + \norm[b,D]{v_h}^2\Bigr)^{\half}
\end{equation*} 
for the respective norm.

For an arbitrary subset \(\omega\subset\fulldomain\), we define \(\Patch{0}{\omega}\) as the smallest subset of \(\Th\) whose union contains~\(\omega\). 
For \(\ell\geq 1\), we define the \(\ell\)-th order patch recursively by
\begin{equation}
    \label{dslp:eq:defPatches}
    \Patch{\ell}{\omega}
    \coloneqq
    \setc{K \in \Th}{\exists\, \widehat{K} \in \Patch{\ell -1}{\omega}
    \text{ such that } K \cap \widehat{K} \neq \emptyset}.
\end{equation}
We illustrate this definition in Figure~\ref{dslp:Fig:patches}.
The \(\ell\)-th boundary layer of cells is denoted by 
\begin{equation}\label{dslp:eq:BoundaryLayer}
    \bLayer{\ell} \coloneqq \Patch{\ell}{\omega} \setminus \Patch{\ell - 1}{\omega}.
\end{equation}
For notational simplicity, we identify cell sets such as \(\Patch{\ell}{\omega}\) and \(\bLayer{\ell}\) with their associated domains, i.e., with the interiors of the unions of their cells.

\begin{figure*}[t]
        \centering
        \includegraphics{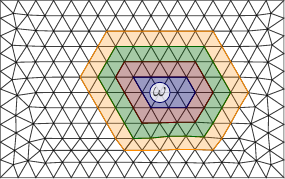}
        \caption{Patches \(\Patch{\ell}{\omega}\) around the subdomain \(\omega \subset \fulldomain\), for \(\ell = 1\) (red), \(\ell = 2\) (green), and \(\ell = 3\) (orange).}
        \label{dslp:Fig:patches}
\end{figure*}

\subsection{Crank--Nicolson time integration scheme}
\label{dslp:subsec:CN}
The classical implicit second-order Crank--Nicolson (CN) scheme serves as the global reference solution in our analysis. We record equivalent formulations that are advantageous for the localized prediction and derive a stability bound in the discrete energy space \(X_h(\fulldomain)\).

To discretize in time, let \(\ts>0\) be the time step size and assume, for simplicity, that \(T=\numberTS\ts\) for some \(\numberTS\in\mathbb N\). 
If necessary, a shorter final time step can be used instead.
We write \(t_n=n\ts\) for \(n=0,\dots,\numberTS\).
We denote the CN approximation at time \(t_n\) by \(\solCN{n} = \pmatT{\uCN{n} \; \vCN{n}} \in X_h\) for~\(n=0,\dots,\numberTS\). The \(n\)-th step of the CN scheme is written in componentwise operator form on~\((\PkTh,\norm[b]{\cdot})\) as 
\begin{subequations}
    \label{dslp:eq:CNfirstOrderFormulation}
    \begin{align}
        \uCN{n} &= \uCN{n-1} + \ts[2] \left( \vCN{n} + \vCN{n-1} \right), \label{dslp:eq:CNfirstOrderFormulation-a}\\
        \vCN{n} &= \vCN{n-1} - \ts[2] L_h \left(  \uCN{n} + \uCN{n-1}\right) + \ts\overline{f}_h^n ,\label{dslp:eq:CNfirstOrderFormulation-b}
    \end{align}
\end{subequations}
with \(\overline{f}_h^n = \frac12 (f_h^n + f_h^{n-1})\).
One way to implement the scheme is to insert
\eqref{dslp:eq:CNfirstOrderFormulation-a} into~\eqref{dslp:eq:CNfirstOrderFormulation-b} to eliminate~\(\vCN{n}\). 
In this way, we may rewrite~\eqref{dslp:eq:CNfirstOrderFormulation} as an operator equation on \((\PkZeroTh,\norm[b]{\cdot})\) for the increment of the first component \(\delta \uCN{n} = \uCN{n} - \uCN{n-1} \in \PkZeroTh\), which yields
\begin{equation}
    \label{dslp:eq:CN-increment-formulation}
    \bigl(\id +\frac{\tau^2}{4} L_h\bigr) \delta \uCN{n}
    =
    \tau\vCN{n-1}
    -\frac{\tau^2}{2} L_h \uCN{n-1}
    +\frac{\tau^2}{2} \overline{f}_h^n.
\end{equation} 
After computing \(\delta \uCN{n}\), we obtain the next iterate via \(\uCN{n} = \uCN{n-1} + \delta \uCN{n}\) and update \(\vCN{n}\) by rearranging~\eqref{dslp:eq:CNfirstOrderFormulation-a}.

Yet another way to rewrite \eqref{dslp:eq:CNfirstOrderFormulation} can be found in \cite[Section 11.1]{DoeHKRSW23}. 
With 
\begin{equation}
    \label{dslp:eq:defCNoperator-R}
    R_{-} = \pmat{\id &-\ts[2]\id \\ \ts[2] L_h & \id}, \quad R_{+} = \pmat{\id &\ts[2]\id \\ -\ts[2] L_h & \id}, \quad \text{and} \quad R = R_{-}^{-1} R_{+},
\end{equation}
the system \eqref{dslp:eq:CNfirstOrderFormulation} can be rewritten in matrix form as
\begin{equation}
    \label{dslp:eq:CN-Rplusminus}
    R_{-} \pmat{\uCN{n} \\ \vCN{n}} = R_{+} \pmat{\uCN{n-1} \\ \vCN{n-1}} + \ts \pmat{0 \\ \overline{f}_h^n}
\end{equation}
for \(n=1,\dots,\numberTS\). This formulation is especially advantageous when investigating the stability of the CN scheme, since 
\begin{equation}
    \label{dslp:eq:CNoperator-R-NormEstimates}
    \norm[X_h \leftarrow X_h][\big]{R_{-}^{-1}}\leq 1 \quad \text{ and } \quad \norm[X_h \leftarrow X_h][\big]{R} = 1, 
\end{equation}
see, e.g., \cite[Section 11.1]{DoeHKRSW23}. Iterating this representation gives the following standard stability estimate. 
\begin{lemma}[CN stability]
    \label{dslp:lem:CNstability}
    Let the data be well-prepared according to Definition~\ref{dslp:def:data-well-prepared}. Then, 
    \begin{equation}
        \label{dslp:lem:CNstability-eq}
        \norm[X_h][\big]{\solCN{n-1}} + \ts \norm[b]{\overline{f}_h^n} \leq \Cdata
    \end{equation}
    for every \(n = 1, \dots, \numberTS\), with \(\Cdata\) defined in \eqref{dslp:eq:def-Cdata}.
\end{lemma}

\begin{proof}
    We use \eqref{dslp:eq:CN-Rplusminus} to expand 
    \(\solCN{n-1} = R^{n-1} \solCN{0}+ \ts  \sum_{j=1}^{n-1} R^{n-1-j} R_{-}^{-1} \pmatT{0  \; \;  \overline{f}_h^j}. \)
	With \eqref{dslp:eq:CNoperator-R-NormEstimates}, we derive 
	\begin{equation}
		\label{dslp:eq:CNstability}
		\norm[X_h][\big]{\solCN{n-1}} \leq \norm[X_h][\big]{\solCN{0}} + \ts \sum_{j=1}^{n-1} \norm[b][\big]{\overline{f}_h^j}.
	\end{equation}
    By \eqref{dslp:eq:CNstability} and the definition of \(\uCN{0}, \vCN{0}\),  and \(\overline{f}_h^j\) for \(j=1,\dots,\numberTS\), we get 
    \begin{align*}
        \norm[X_h][\big]{\solCN{n-1}} + \ts \norm[b][\big]{\overline{f}_h^n} 
        &\leq \norm[X_h][\big]{\solCN{0}}
        + \ts \sum_{j=1}^{n} \norm[b][\big]{\overline{f}_h^j} \\
        &\leq \norm[a]{\Rproj \uExakt[0]} + \norm[b]{\LTproj \vExakt[0]}  
        + T  \max\limits_{j=0,\dots,n} \norm[b][\big]{\LTproj f(t_j)} \\ 
        &\leq \norm[a]{\uExakt[0]} + \norm[b]{\vExakt[0]} 
        + T  \max\limits_{j=0,\dots,n} \norm[b][\big]{f(t_j)}. 
    \end{align*}
    The claim follows directly by the definition of \(\Cdata\), see \eqref{dslp:eq:def-Cdata}.
\end{proof}

\section{Domain splitting with localized implicit prediction}
\label{dslp:sec:method}

In this section, we construct a non-iterative domain decomposition time integrator in the spirit
of \cite{BucH25CG}, but replace the mass lumping-dependent leapfrog prediction
by a localized implicit CN prediction inspired by the localized implicit time stepping scheme from \cite{GalM23}.

The method is based on an overlapping domain decomposition. 
On each overlapping subdomain, we perform one local CN step. 
Before, we predict the artificial boundary values required on the subdomain interfaces.
Each prediction is computed on a local strip consisting of only a few mesh layers around the corresponding interface.
The prediction and the subdomain CN solves are local operations and can be parallelized over the subdomains. 
The final coupling step then restores conformity by a nodal averaging procedure.
Compared with the scheme from \cite{BucH25CG}, replacing the explicit prediction with a localized implicit one means that the analysis no longer requires a CFL-type step size restriction and permits standard conforming finite elements of any fixed order~\(k\geq 1\).

The section is organized as follows. 
We first introduce the overlapping decomposition and the boundary liftings used to handle artificial interface data. 
We then describe the localized implicit prediction and the local CN subdomain solves. 
Finally, we define the averaging step and summarize the full time stepping algorithm, fixing the notation used in the later error analysis.

\paragraph{Overlapping domain decomposition}
Let
\[
    \overline{\fulldomain}
    =
    \bigcup_{i=1}^{\numberSD}\overline{\SD{i}}
\]
be a non-overlapping decomposition of~\(\fulldomain\). For an overlap parameter \(\ovp \in \mathbb{N}\), we
denote by \(\SDov{i}\) the subdomain obtained by extending \(\SD{i}\) by
\(\ovp\) layers of mesh elements, i.e., 
\begin{equation*}
    \SDov{i} = \dom \Patch{\ovp}{\Th[\SD{i}]},
\end{equation*}
with the definition of patches from \eqref{dslp:eq:defPatches}, see also Figure~\ref{dslp:Fig:OverlappingDD}. The physical width of the overlap is denoted by \(\ov\sim h \ovp\). The interface of the overlapping subdomain is given by
\[
    \SDovint{i} \coloneqq \dSDov{i}\cap \fulldomain .
\]

The localized implicit prediction step uses two additional overlap parameters~\(\povp[1],\povp[2] \in\mathbb N\) with the minimal condition \(\povp[1]+\povp[2]>2\) for the method to be well-defined. The parameter~\(\povp[1]\) determines how far information is collected around the interface \(\SDovint{i}\) when forming the localized right-hand side. 
The parameter \(\povp[2]\) determines the size of the prediction domain, so that the artificial boundary conditions of the prediction problem are sufficiently far away from the interface. 
 
\begin{figure}[t]
    \centering
    \includegraphics{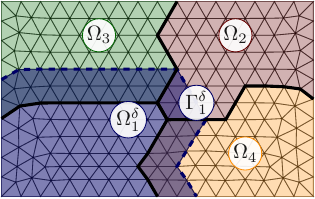}
    \hspace*{1cm}
    \includegraphics{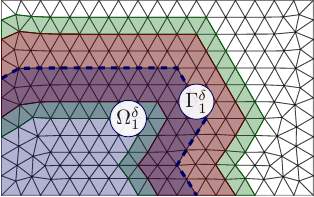}
    \caption{Left: Decomposition into overlapping subdomains \(\SDov{i}\), \(i\in \monosetc{1,\dots,\numberSD}\). Starting from non-overlapping subdomains \(\SD{i}\), we extend each subdomain by \(\ovp\) mesh layers. Here, \(\ovp=2\) and \(\numberSD=4\). Right: Overlap around the prediction interface \(\SDovint{i}\) for \(\povp[1] = 2\) and \(\povp[2] = 1\).}
    \label{dslp:Fig:OverlappingDD}
\end{figure}

\paragraph{Boundary liftings}
Before turning to the description of the non-iterative time integrator, we comment on the lifting property used to handle inhomogeneous interface data in the local subdomain problems.  
The analysis requires that traces prescribed on the artificial boundary \(\SDovint{i}\) can be extended into \(\SDov{i}\) without affecting the physical boundary and with support restricted to a fixed strip near the interface. 
This allows boundary contributions to be estimated by local norms near~\(\SDovint{i}\), uniformly with respect to the global domain.

\begin{assumption}[Stable boundary liftings]
	\label{dslp:ass:boundaryLifting}
	 	For each interface \(\SDovint{i}\), \(i\in\monosetc{1,\dots,\numberSD}\), there exists a linear lifting operator
	\begin{equation*}
		\blifting{i}: \restr{\PkTh}{\SDovint{i}} \to \PkTh[\SDov{i}], 
	\end{equation*}
	such that for every trace datum \(r_i \in \restr{\PkTh}{\SDovint{i}}\),
	\begin{equation*}
		\restr{\blifting{i} r_i}{\SDovint{i}} = r_i, 
		\quad
		\restr{\blifting{i} r_i}{\dSDov{i} \cap \dfulldomain} = 0,
		\quad 
		\supp \blifting{i} r_i \subset \LiftStrip{i} \subset \SDov{i}, 
	\end{equation*}
	where \(\LiftStrip{i}\) is a fixed strip around \(\SDovint{i}\), and 
	\begin{equation}
		\label{dslp:eq:BoundaryLiftingStability}
		\tnorm[\SDov{i}]{\blifting{i} r_i} \leq \Cblift h^{-1/2} \tnorm[\LiftStrip{i}]{r_h}
		\quad
		\text{for all } r_h\in\PkTh[\SDov{i}]
		\text{ with }
		\restr{r_h}{\SDovint{i}}=r_i.
	\end{equation}
    Here, for \(D \subseteq \fulldomain\), the \(\tnorm[D]{\cdot}\)-norm is given by 
    \begin{equation}
        \label{dslp:eq:tnorm}
        \tnorm[D]{u} \coloneqq \Bigl( \norm[a,D]{u}^2 + \frac{4}{\tss}\norm[b,D]{u}^2 \Bigr)^{\half}.
    \end{equation}
    For \(D=\fulldomain\), we again omit the domain index and simply write \(\tnorm{\cdot}\).
\end{assumption}
\begin{remark}[Existence of appropriate liftings]
	Assumption~\ref{dslp:ass:boundaryLifting} is satisfied by standard discrete trace lifting constructions. For example, one may first construct a harmonic lifting in a tubular strip of width \(\varepsilon \in \mathcal{O}(h)\), extend it by zero, and apply Scott--Zhang quasi-interpolation \cite{ScoZha90}, as done in \cite[Appendix~A]{BucD26}.
	Such a construction gives a discrete lifting supported in the strip around \(\SDovint{i}\) and is stable with respect to the standard \(H^{1/2}_{00}(\SDovint{i})\)-trace norm\footnote{The space \(H^{1/2}_{00}(\SDovint{i})\) contains all the functions in \(H^{1/2}(\SDovint{i})\) whose extension by zero has a stable \(H^{1/2}(\partial\SDov{i})\)-trace, see also \cite{LionsMagenes72}.}. To obtain~\eqref{dslp:eq:BoundaryLiftingStability}, one combines this with a local trace estimate bounding the trace norm of 
    \begin{equation*}
        r_i=r_h|_{\SDovint{i}}
    \end{equation*}
    by a local \(H^1\)-norm of \(r_h\) on~\(\LiftStrip{i}\), which is controlled by the \(\tnorm{\cdot}\)-norm on the discrete space. 
    In that case, \(\Cblift\) depends on the shape regularity of the mesh. 
\end{remark}

\paragraph{Description of the method}
For the mathematical description and the subsequent error analysis, it is
convenient to write the method in terms of a global approximation
\[
    \solDS{n}
    =
    \pmat{\uDS{n}\\ \vDS{n}}
\]
at each time step \(t_n = n \tau\) for \(n=0,\dots, \numberTS\).
Note that \(\solDS{n}\) is only a theoretical global object used for the analysis; computationally, it is represented by local restrictions.

With this convention, one step of the method consists of the following three
stages.
\begin{enumerate}[1)]
    \item \textbf{Localized implicit prediction.}
    For each artificial interface \(\SDovint{i}\), a local implicit
    prediction~\(\solLP{i}{n}\) is computed on a strip around
    \(\SDovint{i}\).  Only its displacement component \(\uLP{i}{n}\) is used
    as Dirichlet data for the following subdomain solve.

    \item \textbf{CN solves on overlapping subdomains.}
    On every \(\SDov{i}\), one CN step is performed with initial
    value \(\restr{\solDS{n-1}}{\SDov{i}}\), with homogeneous boundary data on
    \(\dSDov{i}\cap\dfulldomain\), and with the predicted boundary value
    \(\restr{\uLP{i}{n}}{\SDovint{i}}\) on the artificial boundary.  The
    result is denoted by
    \[
        \solDSloc{i}{n}
        =
        \pmat{\uDSloc{i}{n}\\ \vDSloc{i}{n}}.
    \]

    \item \textbf{Averaging.}
    The local solutions are restricted to the non-overlapping subdomains~\(\SDc{i}\).  At nodes shared by several \(\SDc{i}\), their values
    are averaged, which gives the new global approximation
    \(\solDS{n}\). 
\end{enumerate}
The method is non-iterative because each time step uses exactly one prediction,
one set of independent subdomain solves, and one averaging operation.

\subsection{Localized implicit prediction}
\label{dslp:subsec:method:localizedPrediction}

The localized implicit prediction is the new component compared with the domain splitting scheme from \cite{BucH25CG}.
It is built only from data near the interface~\(\SDovint{i}\).
We formulate the localized prediction based on the Crank--Nicolson formulation for the displacement increment similar to \eqref{dslp:eq:CN-increment-formulation}.

For each overlapping interface~\(\SDovint{i}\), the localized implicit prediction approximates the displacement increment produced by a global Crank--Nicolson step \eqref{dslp:eq:CN-increment-formulation}, but using only data and computations from a neighborhood of that interface. 
The construction involves two localization steps:
\begin{enumerate}
    \item A cutoff restricts the right-hand side of the displacement increment problem.
    \item We truncate the domain in which the increment is calculated. 
\end{enumerate}
This localization is done for each \(\SDovint{i}\) individually to obtain the respective boundary data on the interface.
For a fixed interface \(\SDovint{i}\), we take a
cutoff function 
\begin{equation*}
    \chi_i\in\PkTh \quad \text{with} \quad 0\leq \chi_i\leq 1, \quad \text{such that } \chi_i\equiv1 \text{ on the lifting strip \(\LiftStrip{i}\)}
\end{equation*}
from Assumption~\ref{dslp:ass:boundaryLifting}.  
The support of~\(1-\chi_i\) is separated from~\(\LiftStrip{i}\) by~\(\povp[1]\) layers, and there exists a bound \(\norm[L^\infty(\fulldomain)]{\grad\chi_i} \leq c_i h^{-1}\) for a constant \(c_i>0\) only depending on the spatial dimension and the shape regularity of the mesh. We illustrate the choice of the cutoff \(\chi_i\) in Figure~\ref{dslp:Fig:cutoff-chi-i}.

For notational convenience, we set
\[
    D_i^{\mathrm{p}} \coloneqq \dom\Patch{\povp[2]}{\supp\chi_i}.
\]
With
\[
    \mathcal K_{D_i^{\mathrm{p}}}(z_h,w_h)
    \coloneqq
    \abil[D_i^{\mathrm{p}}]{z_h,w_h}
    +
    \frac{4}{\tau^2}\bbil[D_i^{\mathrm{p}}]{z_h,w_h},
\] 
the locally computed
increment \(\widehat{\delta}\widetilde u_{\chi_i}^n\in\PkZeroTh[D_i^{\mathrm{p}}]\) is defined by
\begin{equation}
    \label{dslp:eq:method:localizedPredictionIncrement}
    \begin{aligned}
    \tss[4]\mathcal K_{D_i^{\mathrm{p}}}(\widehat{\delta}\widetilde u_{\chi_i}^n,w_h)
    = \; &
    \tau\bbil[D_i^{\mathrm{p}}]{\vDS{n-1},\NodalInt(\chi_i w_h)} -\frac{\tau^2}{2}\abil[D_i^{\mathrm{p}}]{\uDS{n-1},\NodalInt(\chi_i w_h)}
     \\ 
    &+ \frac{\tau^2}{2}\bbil[D_i^{\mathrm{p}}]{\overline f_h^n,\NodalInt(\chi_i w_h)},
    \end{aligned}
\end{equation}
for all \(w_h\in\PkZeroTh[D_i^{\mathrm{p}}]\). On the boundary of \(D_i^{\mathrm{p}}\), we impose homogeneous boundary
conditions on the increment. Thus, after extending \(\widehat{\delta}\widetilde u_{\chi_i}^n\) by zero
outside \(D_i^{\mathrm{p}}\), the predicted displacement is
\begin{equation}
    \label{dslp:eq:method:predictedDisplacement}
    \uLP{i}{n}
    \coloneqq
    \uDS{n-1}+\widehat{\delta}\widetilde u_{\chi_i}^n .
\end{equation}

Localizing the CN increment by applying the cutoff \(\chi_i\) to the test
function is a new technique compared to  \cite{GalM23}. 
Note, however, that on the full domain \(\fulldomain\), the right-hand sides corresponding to \(\chi_i\) and \(1-\chi_i\) still add up to the original right-hand side of the global CN increment formulation. 
Thus, this localization could also be used in a superposition argument similar to \cite{GalM23}.
We further restrict the computation to the enlarged domain \(D_i^{\mathrm{p}}\) and impose homogeneous boundary conditions on the increment.
Since \(\widehat{\delta}\widetilde u_{\chi_i}^n \in \PkZeroTh[D_i^{\mathrm{p}}]\), the prediction \(\uLP{i}{n} = \uDS{n-1} + \widehat{\delta}\widetilde u_{\chi_i}^n\) coincides with~\(\uDS{n-1}\) on the boundary of \(D_i^{\mathrm{p}}\).
The locally computed increment can therefore be extended by zero beyond~\(D_i^{\mathrm{p}}\) without affecting conformity.
The influence of the artificial boundary condition decays across the~\(\povp[2]\) layers, which motivates this localization.
These two steps, namely splitting the right-hand side by cutoffs for the test functions and then exploiting the decay away from the boundary, are reflected in the later error analysis.

\begin{figure*}[t]
    \centering
    \includegraphics{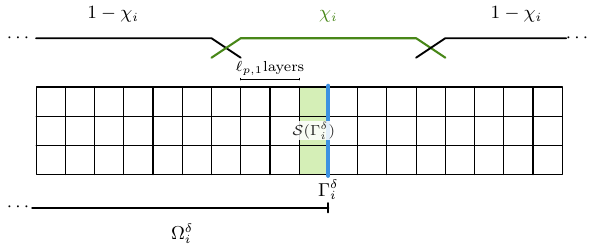}
    \caption{Sketch of the geometric requirements on the cutoff function \(\chi_i\) for a regular rectangular mesh. Here, the lifting strip \(\LiftStrip{i}\) is contained in the first layer of cells next to the interface and \(\povp[1]=2\).} 
    \label{dslp:Fig:cutoff-chi-i}
\end{figure*}

The velocity component \(\vLP{i}{n}\) can be reconstructed from the
CN relation if needed, but the CN step on the subdomains only uses
\(\restr{\uLP{i}{n}}{\SDovint{i}}\).

\begin{remark}[Realization of localized right-hand sides]
    The localized right-hand side in \eqref{dslp:eq:method:localizedPredictionIncrement} does not require a separate assembly procedure. 
    Since \(\chi_i\) and \(w_h\) are finite element functions, nodal interpolation of their product is obtained by pointwise multiplication of their nodal values.
    Hence, one may first assemble the CN increment right-hand side on the full prediction domain \(D_i^{\mathrm{p}}\) and then multiply the
    resulting nodal entries by the respective values of the cutoff function. 
\end{remark}

\subsection{Subdomain CN solves}
With the localized prediction that provides the artificial boundary data, the
remaining computation is a standard CN step on each overlapping subdomain.
For each \(i\in\monosetc{1,\dots,\numberSD}\), we therefore compute
\(\solDSloc{i}{n}\) on \(\SDov{i}\) based on \(\solDS{n-1}\) as the previous approximation and with boundary data prescribed by the
prediction on \(\SDovint{i}\).
That is, \(\solDSloc{i}{n}\) is the result of one step of \eqref{dslp:eq:CNfirstOrderFormulation} on \(\SDov{i}\) with
\begin{equation*}
    \restr{\uDSloc{i}{n}}{\SDovint{i}}
    =
    \restr{\uLP{i}{n}}{\SDovint{i}},
    \qquad
    \restr{\uDSloc{i}{n}}{\dSDov{i}\cap\dfulldomain}
    =
    0.
\end{equation*}
Equivalently, the inhomogeneous trace can be incorporated by the lifting
operator from Assumption~\ref{dslp:ass:boundaryLifting}.  Since the local
problems are posed on the overlapping domains \(\SDov{i}\), all linear systems
are smaller than the global CN system and can be solved in
parallel.  

\subsection{Averaging}
To combine the local functions \(\solDSloc{i}{n}\) from the last step, we first restrict them to the non-overlapping subdomains. 
However, summing these restrictions yields a broken function that is generally not conforming. 
To restore conformity, we average all available nodal values at nodes shared by several subdomains. 
The following definition makes this procedure precise.

Adopting the notation in \cite[Section~19.1]{ErnG21I}, we denote by \(\mathcal{A}_h\) the set of connectivity classes of the Lagrangian nodes of the mesh. 
That is, each class \(a\in\mathcal{A}_h\) is represented by tuples \((K,j)\) that correspond to the same global degree of freedom.
We denote the corresponding global shape function by \(\varphi_a\).

The averaging operator \(\averaging\) maps local functions on the overlapping
subdomains to a global finite element function by its nodal values.  
For local functions \(u_i\) defined on \(\SDov{i}\), for \(i\in\monosetc{1,\dots,\numberSD}\), let \( u_{\fulldomain}^{\#} \coloneqq \sumSD \restr{u_i}{\SD{i}} \) be the corresponding (not necessarily continuous) global function.
We set
\begin{equation}
    \label{dslp:eq:method:averaging}
    \begin{aligned}
         \averaging(\{u_i\}_{i=1}^{\numberSD})
        &\coloneq
        \sum_{a \in \mathcal{A}_h} \frac{1}{\card{(a)}} \sum_{(K, j) \in a} \sigma_{K,j}\Bigl( \restr{u_{\fulldomain}^{\#}}{K}\Bigr) \varphi_a \\
        &=
        \sum_{a \in \mathcal{A}_h} \frac{1}{\card{(a)}} \sum_{(K, j) \in a} \Bigl( \restr{u_{\fulldomain}^{\#}}{K}\Bigr)(a_{K,j}) \; \varphi_a
    \end{aligned}
\end{equation}
    In fact, this definition just encodes the averaging operator \(J_h^\mathrm{av}\) from \cite[Section 22.2]{ErnG21I} applied to the broken polynomial~\(u_{\fulldomain}^{\#}\) and coincides with the averaging in \cite[Section~3.3]{BucH25CG}.
The definition is applied componentwise to solution vectors.  Thus,
\begin{equation}
    \label{dslp:eq:method:globalDSupdate}
    \solDS{n}
    =
    \averaging\bigl(
        \{\restr{\solDSloc{i}{n}}{\SDov{i}}\}_{i=1}^{\numberSD}
    \bigr).
\end{equation}
In particular, if a finite element function is already globally continuous,
then restricting it to the overlapping subdomains and applying
\(\averaging\) leaves it unchanged.

\begin{remark}[Distributed update]
Although \eqref{dslp:eq:method:globalDSupdate} is written in global form, it is only a compact notation for a distributed update.
In practice, each subdomain keeps the values on its non-overlapping part \(\SD{i}\).
Values needed only in the overlap or prediction domain are obtained from the neighboring subdomains that own the corresponding non-overlapping parts.
At the interfaces of non-overlapping subdomains, neighboring subdomains exchange their local nodal values and replace them by the average prescribed in \eqref{dslp:eq:method:averaging}.
Thus, global vectors never need to be assembled.
\end{remark}

\subsection{Full time-stepping algorithm}

Before we can perform the first time step, we need to transfer the initial data to the discrete spaces we work in. Here, we use again the Ritz projection~\(\Rproj\) from \eqref{dslp:eq:RitzProjection} and the \(L^2\)-projection \(\LTproj\) from \eqref{dslp:eq:L2Projection}, so that the initial values of the domain splitting method coincide with the initial global CN approximation, i.e.,  
\begin{equation*}
    \solDS{0} = \pmat{\Rproj \uExakt[0] \\ \LTproj \vExakt[0]} = \solCN{0}.
\end{equation*} 
Other, in particular local, choices are possible. We refer, e.g., to the analysis in \cite{BucH25CG}, where we initialized with a nodal interpolation \(\NodalInt\) in both components.
The specific choice of the Ritz projection is advantageous for the CN stability result and simplifies the overall analysis. This is, however, not a restriction.
After this first projection, we provide each subdomain with the local restrictions
\(\restr{\solDS{0}}{\SDovc{i}}\) and
\(\restr{\solDS{0}}{\overline{D}_i^{\mathrm{p}}}\).
For a later time step \(n\), the available data is given by 
\(\restr{\solDS{n-1}}{\SDovc{i}}\) and
\(\restr{\solDS{n-1}}{\overline{D}_i^{\mathrm{p}}}\) instead.

With these local restrictions, the prediction and subdomain CN solves can be assembled as described above.
We summarize the full non-iterative time integrator in Algorithm~\ref{dslp:alg:DSLP}.

\begin{algorithm}[ht!]
    \SetKwFor{ParFor}{for}{do \normalfont{(in parallel)}}{end}
    \DontPrintSemicolon
    \SetAlgoLined
    \SetKwInput{Input}{input}
    \SetKwInOut{Output}{result}
    \Input{local initial data
        \(\{\restr{\solDS{0}}{\SDovc{i}}\}_{i=1}^{\numberSD}\)
        and the corresponding restrictions on the prediction patches,
        overlapping decomposition
        \(\overline{\fulldomain}=\bigcup_{i=1}^{\numberSD}\SDovc{i}\),
        time step size \(\tau\), final time \(T=\numberTS\tau\),
        prediction parameters \(\povp[1]\), \(\povp[2]\)}
    \(n=1\)\;
    \While{\(n\leq \numberTS\)}{
        \ParFor{\(i=1,\dots,\numberSD\)}{
            choose the interface cutoff \(\chi_i\) around \(\SDovint{i}\)\;
            compute \(\widehat{\delta}\widetilde u_{\chi_i}^n\) from
            \eqref{dslp:eq:method:localizedPredictionIncrement} on
            \(D_i^{\mathrm{p}}\) using the local
            restriction of \(\solDS{n-1}\)\;
            set \(\uLP{i}{n}=\uDS{n-1}
            +\widehat{\delta}\widetilde u_{\chi_i}^n\)\;
        }
        \ParFor{\(i=1,\dots,\numberSD\)}{
            compute \(\solDSloc{i}{n}\) by one CN step on
            \(\SDov{i}\) with boundary data
            \(\restr{\uLP{i}{n}}{\SDovint{i}}\) on \(\SDovint{i}\)\;
        }
        average nodal values on the non-overlapping interfaces according to
        \eqref{dslp:eq:method:averaging}\;
        exchange neighboring data to update
        \(\restr{\solDS{n}}{\SDovc{i}}\) and \(\restr{\solDS{n}}{\overline{D}_i^{\mathrm{p}}}\) for
        the next step\;
        update \(n\leftarrow n+1\)\;
    }
    \Output{
    \(\solDS{\numberTS}\approx \solutionVector(\cdot,T)\)}
    \caption{Domain splitting with localized implicit prediction}
    \label{dslp:alg:DSLP}
\end{algorithm}

\subsection{Approximations within the error analysis}

Before we present the error analysis of our scheme, we gather the notation for the different relevant approximations.
The solution vectors \(\solCNmod{n}\) and~\(\solLP{i}{n}\) are both one-step objects obtained from the same domain splitting approximation~\(\solDS{n-1}\) in the previous step. 
The former is computed on the full domain and used only for theoretical purposes, whereas the latter is computed only near the interface~\(\SDovint{i}\) and provides the boundary data for the local subdomain solve on~\(\SDov{i}\). These local subdomain solves yield the local subdomain solutions~\(\solDSloc{i}{n}\), which are finally averaged to obtain the next domain splitting approximation~\(\solDS{n}\).

\section{Main Results}
\label{dslp:sec:main-results}
In this section, we present the main result of this work, which shows convergence of the domain splitting method with localized implicit prediction from Section~\ref{dslp:sec:method} against the global CN approximation, i.e., 
\begin{equation*}
	\norm[X_h][\big]{\solDS{n}-\solCN{n}} \in \mathcal{O}(\tau^p)
\end{equation*} 
for some \(p > 0\), which depends on the parameter choices within the method.
Instead of a typical CFL condition, which requires a time step restriction of the form \(\tau \leq c h\) for some constant \(c>0\), we pose a condition on the overlap parameters \(\ovp, \povp[1]\) and \(\povp[2]\) ensuring a suitable accuracy for a desired order \(p\), whereas stability is always satisfied. This new result applies to standard conforming finite elements of arbitrary fixed order \(k\geq 1\) and includes the material coefficients \(A\) and \(B\) in the analysis.

The convergence proof is organized around three main ingredients. 
First, we bound the localized prediction error on the artificial interfaces. 
To this end, we adapt localization and decay ideas inspired by \cite{GalM23} to CN increments. 
A central point is that the localization is imposed through the test functions, which allows us to control the resulting prediction error in the \(X_h\)-based framework. 
This novel localization approach requires several adjustments compared to \cite{GalM23}.
Second, the localized prediction error has to be propagated through the local subdomain solves.
We estimate how this boundary error enters the subdomain error by means of a discrete Saint-Venant principle from \cite{BucD26}.
Here, we also include the additional material parameter \(B\) and combine the result with the lifting stability from Assumption~\ref{dslp:ass:boundaryLifting}.
Finally, we control the error introduced by the nodal averaging step. 
In contrast to \cite{BucH25CG}, the averaging stability cannot rely on a CFL condition.
Hence, factors and constants depending on \(\tau/h\) have to be tracked explicitly throughout the argument.
To combine these ingredients, we use the domain splitting error mechanism from \cite{BucH25CG}, again with modifications due to the material parameters \(A\) and \(B\), the higher-order finite element discretization, and the absence of a CFL condition.

The presented fully discrete convergence analysis against the CN approximation can be extended to error bounds against the exact solution, by also accounting for the error of the global CN method. The corresponding error analysis follows the same lines as in \cite[Theorem~4.2]{BucH25CG}.

Let the coefficients \(A,B\in L^\infty(\fulldomain)\) satisfy \(\alpha \leq A(x), B(x) \leq \beta\), as introduced in Section~\ref{dslp:sec:introduction}. The mesh family \(\Th[\fulldomain]\) is affine, matching, shape-regular, and we consider a conforming finite element space discretization with fixed polynomial degree \(k\geq1\). 
The domain decomposition has a fixed finite number \(\numberSD\) of subdomains, overlap of \(\ovp\) mesh layers, and prediction overlap parameters~\(\povp[1], \povp[2]\).

We now state the rather technical convergence result, which in essence shows that the difference between the domain splitting approximation and the global CN approximation decays exponentially as the overlap parameters \(\ovp, \povp[1], \povp[2]\) increase.
The result does not rely on a CFL-type step-size restriction and reads as follows:

\begin{theorem}[Convergence under an overlap condition]
	\label{dslp:thm:ConvergenceOverlapCondition}
    Let \(\uExakt[0] \in H_0^1(\fulldomain)\), \(\vExakt[0] \in L^2(\fulldomain)\) and~\(f \in C([0,T];L^2(\fulldomain))\) be \textit{well-prepared} according to Definition~\ref{dslp:def:data-well-prepared} and let Assumption~\ref{dslp:ass:boundaryLifting} hold.

    Then, the fully discrete domain splitting error satisfies
    \begin{equation}
		\label{dslp:eq:AbstractConvergenceTheoremBound}
        E^n = \norm[X_h][\big]{\solDS{n}-\solCN{n}} \leq C_{\mathrm{data}} n\vartheta \exp(n\vartheta) \qquad \text{for all } n = 1,\dots,\numberTS
    \end{equation}
    with \(C_{\mathrm{data}}\) from Definition~\ref{dslp:def:data-well-prepared} and 
    \begin{equation*}
        \vartheta = \mathfrak{C}_{\tau,h} \numberSD^{\half} \rho^{\ovp}\bigl(\gamma^{\povp[1]} + \gamma^{\povp[2]}\bigr).
    \end{equation*}
	The factor \(\mathfrak{C}_{\tau,h} >0\) depends on \(\tau\) and \(h\) but does not depend on the overlap parameters \(\ovp, \povp[1], \povp[2]\). 
	Further, \(\mathfrak{C}_{\tau,h}\) depends on \(k\), the spatial dimension, the material bounds \(\alpha, \beta\), and the shape regularity of the mesh.
	Both \(\rho\) and \(\gamma\) are contraction factors with \(\rho, \gamma < 1\).

	If \(\ovp,\povp[1],\povp[2]\) are chosen large enough such that \(\vartheta\leq \tau^{p+1}\) for some \(p>0\), it follows that
    \begin{equation}
		\label{dslp:eq:ConvergenceTheoremBound}
		E^n
		\leq \tau^{p}
		C_{\mathrm{data}}
		\,T
		\exp\!\left(T\tau^{p}\right),
	\end{equation}
    i.e., the domain splitting error with respect to the global CN approximation is
	\(\mathcal O(\tau^{p})\). 
\end{theorem}

Note that the overlap condition
\begin{equation*}
    \vartheta\leq \tau^{p+1} \quad \text{for some } p>0
\end{equation*}
in Theorem~\ref{dslp:thm:ConvergenceOverlapCondition} has a simple interpretation and balances two effects.
The prediction overlap has to gather enough information along the relevant wave characteristics so that the localized implicit prediction is accurate at the artificial interfaces.
Thus, we typically choose \(\povp[1]\) and \(\povp[2]\) of the order~\(\tau /h \), which is theoretically and physically reasonable.
The subdomain overlap then further damps the remaining interface error before the next time step. 
This is typically already achieved with a fixed small number of overlap layers.
Together, these requirements ensure that the localized prediction is sufficiently accurate.
We refer to Section~\ref{dslp:sec:discussion} for a more detailed discussion of the overlap condition and the scaling with respect to \(\tau\) and \(h\).

The proof of Theorem~\ref{dslp:thm:ConvergenceOverlapCondition} is presented in Section~\ref{dslp:subsec:main-result-proof}.
Before that, we collect the main ingredients in Sections~\ref{dslp:subsec:main-result-prediction}-\ref{dslp:subsec:main-result-averaging} below.

\subsection{Prediction estimate on the lifting strip}

The first ingredient is a bound for the prediction error made at the interface \(\SDovint{i}\). 
More precisely, we estimate the difference between the localized implicit prediction \(\uLP{i}{n}\) and \(\uCNmod{n}\).
Here, \(\uCNmod{n}\) is the \(u\)-component of a global CN step started from the same previous domain splitting approximation \(\solDS{n-1}\).
Note that we only need to bound the resulting error in the \(u\)-component on the lifting strip \(\LiftStrip{i}\), the fixed strip around \(\SDovint{i}\) from Assumption~\ref{dslp:ass:boundaryLifting}, because the boundary condition is only formulated in the first component.
To derive the following estimate, we introduce a cutoff function \(\chi_i\) and use localization techniques similar to those from \cite{GalM23}.

\label{dslp:subsec:main-result-prediction}
\begin{lemma}[Localized prediction error at the interface]
	\label{dslp:lem:TestLocalizedPredictionEstimateNew}
	Let \(\LiftStrip{i}\subset \SDov{i}\) be the strip from
	Assumption~\ref{dslp:ass:boundaryLifting}.  
	Let \(\uLP{i}{n}\) be the localized prediction as defined in \eqref{dslp:eq:method:predictedDisplacement} with the cutoff \(\chi_i\) satisfying the conditions stated in Section~\ref{dslp:subsec:method:localizedPrediction}.
	Then, the prediction error at the interface \(\SDovint{i}\) measured in the~\(\tnorm{\cdot}\)-norm defined in \eqref{dslp:eq:tnorm} can be bounded by
	\begin{equation}
		\label{dslp:eq:TestLocalizedPredictionEstimateNew}
		\tnorm[\LiftStrip{i}]{\uLP{i}{n}-\uCNmod{n}}
		\leq
				\CLoc
		\left(\gamma^{\povp[1]}+\gamma^{\povp[2]}\right)
		\left(
			\norm[X_h(\fulldomain)][\big]{\solDS{n-1}}
			+\tau\norm[b, \fulldomain]{\overline{f}_h^n}
		\right).
	\end{equation}
		Here, the factor \(\CLoc\) depends only on the bounds of the coefficients~(\(\alpha\) and~\(\beta\)), the shape regularity of the mesh, the spatial dimension, and at most quadratically on the ratio \(\max \monosetc{\tau/h,1}\), but not on the overlap parameters. 
\end{lemma}
The result quantifies the error introduced by the localized implicit prediction. It shows that the difference from a global CN approximation is small on the lifting strip around the interface \(\SDovint{i}\), and decreases exponentially with the prediction overlap widths \(\povp[1]\) and \(\povp[2]\).
The proof is based on localization arguments similar to techniques from \cite{GalM23} and carried out at the end of Section~\ref{dslp:sec:localizedPrediction}.

\subsection{Discrete Saint-Venant principle}

The second ingredient is an application of a discrete Saint-Venant principle (cf.~\cite[Theorem~3.1]{BucD26}) to the difference between the local subdomain solution~\(\uDSloc{i}{n}\) and the global CN approximation~\(\uCNmod{n}\) (again obtained from~\(\solDS{n-1}\) in the previous step). 
Here, we use the boundary lifting assumption, see Assumption~\ref{dslp:ass:boundaryLifting}, to represent the predicted artificial boundary value by the lifting on \(\LiftStrip{i}\).

\label{dslp:subsec:main-result-saintvenant}
\begin{lemma}[Local subdomain error]
	\label{dslp:lem:BoundLocalDiff}
	Let Assumption~\ref{dslp:ass:boundaryLifting} hold.
	Then, the local difference on each non-overlapping subdomain \(\SD{i}\) satisfies 
	\begin{equation*}
		\tnorm[\SD{i}][\big]{\uDSloc{i}{n} - \uCNmod{n}}
		\leq 2 \rho^{\ovp} 
		\tnorm[\SDov{i}][\big]{\blifting{i}\bigl(\restr{\uLP{i}{n}-\uCNmod{n}}{\SDovint{i}}\bigr)}
		\leq 2 \Cblift h^{-\half} \rho^{\ovp} \tnorm[\LiftStrip{i}][\big]{\uLP{i}{n}-\uCNmod{n}}
	\end{equation*}
	with
	\begin{equation*}
		\rho = \Bigl(\frac{M_{\tau/h}}{1+M_{\tau/h}}\Bigr)^{\half}, \quad \text{ and } \quad M_{\tau/h} = C   \frac{\beta^{\half}}{\alpha^{\half}} \Bigl( 1 + \frac{\tau}{2 h}\Bigr),
	\end{equation*}
	Here, \(C > 0\) depends only on the polynomial degree \(k\), the spatial dimension and the shape regularity of the mesh \(\Th\). It is independent of \(h,\,\ts\), and the material parameters \(A\) and \(B\).
\end{lemma}
The estimate in Lemma~\ref{dslp:lem:BoundLocalDiff} follows by applying \cite[Theorem~3.1]{BucD26} to the local difference \(\uDSloc{i}{n}-\uCNmod{n}\) and combining the result with \eqref{dslp:eq:BoundaryLiftingStability}.
Although the estimate is stated in the elliptic \(\tnorm{\cdot}\)-norm of the \(u\)-component, the local error \(\solDSloc{i}{n}-\solCNmod{n}\) has a special structure that later lets us recover control in the full \(X_h\)-norm, meaning that
\begin{equation}
	\label{dslp:eq:NormEqualityLocalDifference}
	\norm[X_h(\SD{i})][\big]{\solDSloc{i}{n} - \solCNmod{n}} =  \tnorm[\SD{i}][\big]{\uDSloc{i}{n} - \uCNmod{n}} \; .
\end{equation}
In Section~\ref{dslp:sec:discreteSaintVenant}, we derive Lemma~\ref{dslp:lem:BoundLocalDiff} in detail, verify the necessary conditions that \(\uDSloc{i}{n} - \uCNmod{n}\) has to satisfy and justify~\eqref{dslp:eq:NormEqualityLocalDifference}.

\subsection{\texorpdfstring{Averaging stability in \(X_h\)}{Averaging stability in the discrete energy space}}
\label{dslp:subsec:main-result-averaging}

The third ingredient for Theorem~\ref{dslp:thm:ConvergenceOverlapCondition} is a stability estimate for the averaging operator \(\averaging\) defined in \eqref{dslp:eq:method:averaging} in the discrete energy norm \(\norm[X_h]{\cdot}\). 
The presented result is a slight modification of \cite[Lemma~7.1]{BucH25CG}, where we used a CFL-type condition to omit the \(\tau / h\) ratio. Instead, we now keep this ratio and include it in the later estimate within the proof of Theorem~\ref{dslp:thm:ConvergenceOverlapCondition}.

\begin{lemma}[Averaging estimate for local errors]
		\label{dslp:Lem:AvgStabilityXh}
		The difference \(\solDS{n}- \solCNmod{n}\) is bounded in the \(X_h(\fulldomain)\)-norm by
		\begin{equation}
			\label{dslp:eq:avg_prop3}
			\begin{aligned}
				\norm[X_h(\fulldomain)][\big]{\solDS{n}- \solCNmod{n}}^2
                 \leq \cavg \Cinv \frac{\beta}{\alpha}\Bigl(  \tss[h^2] +1 \Bigr)  \sumSD \norm[X_h(\SD{i})][\big]{\solDSloc{i}{n}- \solCNmod{n}}^2 , 
			\end{aligned}
		\end{equation}
		with a constant \(\cavg > 0\) and \(\Cinv\) from Lemma~\ref{dslp:Lem:InverseEstimate}, which are both independent of \(h\) and \(\ts\).
\end{lemma}

In essence, Lemma~\ref{dslp:Lem:AvgStabilityXh} is an averaging stability result tailored to the local error structure of the method. It uses the relation between the local \(u\)- and \(v\)-errors to control the global \(X_h\)-error after averaging.
As mentioned above, the proof of this result is motivated by \cite[Lemma~7.1]{BucH25CG}.
It will be presented in detail in Section~\ref{dslp:sec:averaging}.

\subsection{Proof of the convergence result}
\label{dslp:subsec:main-result-proof}

We now have the three ingredients needed to prove Theorem~\ref{dslp:thm:ConvergenceOverlapCondition}.
In essence, the following proof combines Lemmas~\ref{dslp:lem:TestLocalizedPredictionEstimateNew}, \ref{dslp:lem:BoundLocalDiff}, and \ref{dslp:Lem:AvgStabilityXh}. 
The prediction estimate controls the interface error, the discrete Saint-Venant estimate transfers this control to the subdomain interiors, and the averaging estimate turns the local bounds into a global \(X_h\)-error bound.

\begin{proof}[Proof of Theorem~\ref{dslp:thm:ConvergenceOverlapCondition}]
    Let \(\solCNmod{n}\) be the global CN approximation after one time step using~\(\solDS{n-1}\) as previous approximation and with the same right-hand side \(f\). 
	By the triangle inequality, we have
	\begin{equation}
		\label{dslp:eq:TriangleInequality}
		E^n \coloneqq \norm[X_h][\big]{\solDS{n} - \solCN{n}}
		\leq
		\norm[X_h][\big]{\solDS{n}-\solCNmod{n}}
		+
		\norm[X_h][\big]{\solCNmod{n}-\solCN{n}} .
	\end{equation}
	With the CN propagation operator \(R\) defined in \eqref{dslp:eq:defCNoperator-R}, we can rewrite the second difference as
	\begin{equation*}
		\solCNmod{n}-\solCN{n} = R\bigl( \solDS{n-1} - \solCN{n-1}\bigr). 
	\end{equation*}  
	Since the CN propagation operator \(R\) is unitary in the \(X_h\)-norm, this implies
	\begin{equation}
		\label{dslp:eq:estimate-TermB}
		\norm[X_h][\big]{\solCNmod{n}-\solCN{n}}
		\leq \norm[X_h][\big]{\solDS{n-1} - \solCN{n-1}}.
	\end{equation}
	For the difference between \(\solDS{n}\) and \(\solCNmod{n}\), we first use Lemma~\ref{dslp:Lem:AvgStabilityXh} and \eqref{dslp:eq:NormEqualityLocalDifference} to obtain 
	\begin{equation*}
		\begin{aligned}
			\norm[X_h][\big]{\solDS{n}-\solCNmod{n}}^2
			&\leq \cavg \Cinv \frac{\beta}{\alpha}\Bigl(  \tss[h^2] +1 \Bigr)  \sumSD \norm[X_h(\SD{i})][\big]{\solDSloc{i}{n}- \solCNmod{n}}^2 \\
			&= \cavg \Cinv \frac{\beta}{\alpha}\Bigl(  \tss[h^2] +1 \Bigr)  \sumSD \tnorm[\SD{i}][\big]{\uDSloc{i}{n} - \uCNmod{n}}^2
		\end{aligned}
	\end{equation*}
		On each subdomain \(\SD{i}\), \(i\in\monosetc{1,\dots,\numberSD}\), we now apply Lemma~\ref{dslp:lem:BoundLocalDiff} and Lemma~\ref{dslp:lem:TestLocalizedPredictionEstimateNew}, which together yield
	\begin{equation*}
		\tnorm[\SD{i}][\big]{\uDSloc{i}{n} - \uCNmod{n}} \leq 
			2 \Cblift h^{-\half} \rho^{\ovp} \CLoc
		\left(\gamma^{\povp[1]}+\gamma^{\povp[2]}\right)
		\left(
			\norm[X_h(\fulldomain)][\big]{\solDS{n-1}}
			+\tau\norm[b,\fulldomain]{\overline{f}_h^n}
		\right)\; .
	\end{equation*}
	To bound against the data, we apply the CN stability bound from Lemma~\ref{dslp:lem:CNstability} and get
	\begin{equation*}
		\begin{aligned}
			\norm[X_h][\big]{\solDS{n-1}}
			+\tau\norm[b]{\overline{f}_h^n} 
			\leq \norm[X_h][\big]{\solDS{n-1}- \solCN{n-1}} + \norm[X_h][\big]{\solCN{n-1}} +\tau\norm[b]{\overline{f}_h^n} 
			\leq E^{n-1} + C_{\mathrm{data}}, 
		\end{aligned}
	\end{equation*}
	with \(C_{\mathrm{data}} = C_{\mathrm{data}}(\uExakt[0], \vExakt[0], f, T)\).
	Applying these estimates for every interface \(\SDovint{i}\) separately yields
	\begin{equation}
			\label{dslp:eq:estimate-TermA}
			\norm[X_h][\big]{\solDS{n}-\solCNmod{n}}
			\leq
			\mathfrak{C}_{\tau,h} \numberSD^{\half} \rho^{\ovp}\bigl(\gamma^{\povp[1]} + \gamma^{\povp[2]}\bigr)
			\left(
				E^{n-1} + C_{\mathrm{data}}
			\right),
	\end{equation}
	where
	\begin{equation}
		\label{dslp:eq:factor-mathfrakC}
		\mathfrak{C}_{\tau,h} \coloneqq (\cavg \Cinv)^{\half} \frac{\beta^{\half}}{\alpha^{\half}}\Bigl(  \ts[h] +1 \Bigr) 2 \Cblift h^{-\half} \CLoc. 
	\end{equation}
	Recall that \(\CLoc\) depends at most quadratically on \(\tau/h\).
	By combining \eqref{dslp:eq:estimate-TermA} and \eqref{dslp:eq:estimate-TermB} with \eqref{dslp:eq:TriangleInequality}, we get
	\begin{equation}
		\label{dslp:eq:errorDSCNn}
		E^n = \norm[X_h][\big]{\solDS{n}-\solCN{n}} \leq
			(1+\vartheta)E^{n-1}
			+\vartheta
			C_{\mathrm{data}}, 
	\end{equation}
	with \(\vartheta = \mathfrak{C}_{\tau,h}  \numberSD^{\half} \rho^{\ovp}\bigl(\gamma^{\povp[1]} + \gamma^{\povp[2]}\bigr) \).
	Since the initial values \(\solCN{0}\) and \(\solDS{0}\) coincide, the initial error vanishes, i.e., \(E^0=0\), and the iterated use of~\eqref{dslp:eq:errorDSCNn} gives
	\begin{equation}
		\label{dslp:eq:convergenceBoundRecursion}
		E^n 
		\leq \bigl(1 + \vartheta\bigr)^n E^0 + \sum_{k=0}^{n-1} (1+ \vartheta)^k  \vartheta C_{\mathrm{data}} = C_{\mathrm{data}}  \bigl((1+ \vartheta)^n -1\bigr) .
	\end{equation}
	Finally, we use that
	\[
		(1+\vartheta)^n-1
		\leq
		n\vartheta\exp(n\vartheta)
	\]
	to obtain \eqref{dslp:eq:AbstractConvergenceTheoremBound}.
	Choosing \(\ovp,\povp[1],\povp[2]\) such that \(\vartheta\leq \tau^{p+1}\) for some \(p>0\) further yields 
	\begin{equation*}
		(1+\vartheta)^n-1 \leq n\vartheta\exp(n\vartheta)
		\leq
		T\tau^{p}\exp\!\left(T\tau^{p}\right), 
	\end{equation*}
	since \(n\tau\leq T\). 
	Together with \eqref{dslp:eq:convergenceBoundRecursion}, this proves
	\eqref{dslp:eq:ConvergenceTheoremBound}.
\end{proof}

The following three sections prove the main ingredients used above, each
addressing one distinct part of the convergence argument.
Section~\ref{dslp:sec:localizedPrediction} derives the localized prediction estimate at the artificial interfaces. 
Section~\ref{dslp:sec:discreteSaintVenant} shows how this boundary error propagates into the local subdomain solves by means of the discrete Saint-Venant principle. 
Finally, Section~\ref{dslp:sec:averaging} proves the stability of the averaging operator in the discrete energy space \(X_h\).

\section{Localized implicit prediction} 
\label{dslp:sec:localizedPrediction}

\subsection{CN with localized data}
\label{dslp:subsec:localizedPrediction:CNlocalized}
In this section, we investigate properties of the CN approximation~\(\solCN{n}\) at time \(t_n\) under the assumption of localized data at time~\(t_{n-1}\). 
More precisely, we assume that~\(\solCN{n-1}\) and later also the right-hand side \(f(t)\) for \(t\in[t_{n-1}, t_n]\) are supported in \(\omega \subset \fulldomain\). 
We derive bounds for \(\solCN{n}\) in suitable norms away from \(\omega\), i.e., on \(\fulldomain \setminus \omega\). These bounds are the key ingredients for the prediction error estimate of the localized implicit prediction in Section~\ref{dslp:subsec:localizedPrediction:EstimatesPrediction}.
These results are motivated by previous work from \cite{GalM23}. However, the transfer to the present CN setting with localized data requires several adjustments, which is why we present the arguments from scratch.

First, we show that we can switch between the \(\tnorm[\fulldomain \setminus \omega]{\cdot}\)-norm of the \(u\)-component and the energy norm~\(\norm[X_h(\fulldomain \setminus \omega)]{\cdot}\) of the full approximation vector, when the underlying data is supported in \(\omega\). 

\begin{lemma}[CN elliptic control away from data]
	\label{dslp:lem:CN-NormChangeAbsentData}
		Assume \(\uCN{n-1}, \vCN{n-1}\) to be supported in \(\omega \subset \fulldomain\). Then, the CN approximation \(\solCN{n} = \pmatT{\uCN{n} \; \vCN{n}}\) satisfies
		\begin{equation*}
			\tnorm[\fulldomain \setminus \omega]{\uCN{n}} \leq \norm[X_h(\fulldomain \setminus \omega)]{\solCN{n}}.
		\end{equation*}
\end{lemma}

\begin{remark}[No uniform norm change for arbitrary discrete functions]
Let us emphasize that a norm change between \(\tnorm{\cdot}\) and \(\norm[X_h]{\cdot}\) as in Lemma~\ref{dslp:lem:CN-NormChangeAbsentData} is not possible for general \(u_h\in \PkZeroTh[]\).
Recall that
\[
    \tnorm{u_h}^2
    =
    \norm[a]{u_h}^2 + \frac{4}{\tss}\norm[b]{u_h}^2 .
\]
Bounding \(\tnorm{u_h}\) by the \(X_h\)-norm of an associated vector \(x_h=\pmatT{u_h\;v_h}\) would require controlling the weighted \(L^2\)-term by the \(\norm[a]{\cdot}\)-norm of \(u_h\).
A Poincaré-type estimate yields, at best,
\[
    \norm[a]{u_h} \leq \tnorm{u_h} \leq c_a \ts^{-1}\norm[a]{u_h},
\]
with a constant \(c_a>0\) independent of \(\tau\), so the resulting factor \(\ts^{-1}\) would spoil the subsequent estimates for small \(\tau\). The preceding lemma avoids this loss by exploiting the CN relation between the two components of the full vector \(\solCN{n}\).
\end{remark}

\begin{proof}[Proof of Lemma~\ref{dslp:lem:CN-NormChangeAbsentData}]
	Recall that
	\begin{equation}
		\label{dslp:tnormExterioromega}
		\tnorm[\fulldomain \setminus \omega]{\uCN{n}} = \Bigl( \norm[a, \fulldomain \setminus \omega]{\uCN{n}}^2 + \frac{4}{\tss}\norm[b, \fulldomain \setminus \omega]{\uCN{n}}^2 \Bigr)^{\half}. 
	\end{equation}
	Using the CN scheme, we can rewrite 
	\begin{equation*}
		\begin{aligned}
				\norm[b,\fulldomain \setminus \omega]{\uCN{n}} 
				&= \norm[b, \fulldomain \setminus \omega][\big]{\uCN{n-1} + \ts[2] (\vCN{n} + \vCN{n-1})} \\
				&\leq \norm[b,\fulldomain \setminus \omega]{\uCN{n-1}} + \ts[2] \norm[b,\fulldomain \setminus \omega]{\vCN{n-1}} 
				+ \ts[2] \norm[b,\fulldomain \setminus \omega]{\vCN{n}} \\
				&= \ts[2] \norm[b,\fulldomain \setminus \omega]{\vCN{n}}, 
		\end{aligned}
	\end{equation*}
	since \(\uCN{n-1}, \vCN{n-1}\) are supported in \(\omega\). Thus, we obtain with \eqref{dslp:tnormExterioromega}
	\begin{equation*}
	\tnorm[\fulldomain \setminus \omega]{\uCN{n}} \leq \Bigl( \norm[a, \fulldomain \setminus \omega]{\uCN{n}}^2 + \norm[b,\fulldomain \setminus \omega]{\vCN{n}}^2 \Bigr)^{\half} = \norm[X_h(\fulldomain \setminus \omega)]{\solCN{n}}.	
	\qedhere
	\end{equation*}
\end{proof}

The following lemma gives a stability bound for nodal interpolation applied to products of functions in the discrete space \(\PkTh\).
The result is also stated in \cite[Lemma~4.1]{BucD26}, but we include it here because it is used repeatedly below. 

\begin{lemma}[Interpolation stability for discrete products]
    \label{dslp:Lem:StabilityIhProductIh}
    Let \(K \in \Th\). For \(m=0,1\) there exists a constant \(\Cint\) uniform in \(K\) and \(h\), such that
    \begin{equation*}
        \abs{\NodalInt(\phi_h \psi_h)}_{H^{m}(K)} \leq \abs{\phi_h \psi_h}_{H^m(K)} + \abs{(\id -\NodalInt)(\phi_h \psi_h)}_{H^{m}(K)} \leq \Cint \abs{\phi_h \psi_h}_{H^m(K)},
    \end{equation*}
    for all \(\phi_h, \psi_h \in \PkTh\). The constant \(\Cint\) depends on the polynomial degree \(k\), the spatial dimension~\(d\) and the shape regularity of the mesh \(\Th\). For convenience, we denote here \(H^0(K) \coloneqq L^2(K)\).
\end{lemma}
\begin{proof}
	For the second inequality, see \cite[Lemma 4.1]{BucD26}. The first holds by the triangle inequality.
\end{proof}

The next result is a decay estimate for the CN approximation away from the support of the data and based on~\cite{GalM23}. It is analogous to decay estimates in the context of multiscale methods, where local right-hand sides generate solutions whose energy decreases across successive element layers; see, e.g., \cite{MalP21,MalP14}. Similar ideas are also used in \cite{AltHP20}. 

\begin{lemma}[CN decay for localized data]
	\label{dslp:lem:CNlocalizedData}
	Assume \(\uCN{n-1}, \vCN{n-1}, f_h^n, f_h^{n-1}\) to be supported in \(\omega \subset \fulldomain\). Then, for \(\ell \in \mathbb{N}_0\), we have
	\begin{equation*}
        \tnorm[\fulldomain \setminus \Patch{\ell}{\omega}]{\uCN{n}} \leq \gamma^{\ell} \tnorm[\fulldomain \setminus \omega]{\uCN{n}} \leq \gamma^{\ell} \norm[X_h(\fulldomain \setminus \omega)]{\solCN{n}} ,
	\end{equation*}
	with the contraction factor 
	\begin{equation}
		\label{dslp:eq:CNLocalizationConstant}
		\gamma = \Bigl(\frac{W_{\tau/h}}{1 + W_{\tau/h}}\Bigr)^{1/2} < 1,
		\qquad
		W_{\tau/h} = C\Bigl(1+\frac{\tau}{h}\Bigr),
	\end{equation}
    where the constant \(C>0\) depends only on the material contrast \(\frac{\beta^{\half}}{\alpha^{\half}}\), the polynomial degree \(k\), the spatial dimension, and the shape regularity of the underlying mesh.
\end{lemma} 

\begin{proof}
	The case \(\ell = 0\) holds trivially. For \(\ell \geq 1\), let \(D \subseteq \fulldomain\) match \(\Th\) in the sense of \eqref{dslp:eq:D-matching-Th}. 
	For \(u_h \in \PkTh[], \phi \in H^1(\fulldomain)\) we write
	\begin{align}
		\label{dslp:eq:mathcalK-Def}
		\mathcal{K}_{D}(u_h, \phi)
		&\coloneqq  \sum_{K \in \Th[D]} \abil[K]{u_h, \phi} + \frac{4}{\tss} \bbil[K]{u_h, \phi}.
	\end{align}
    Next, we define a cutoff \(\cutoff{\ell} \in \PkTh\), \(0 \leq \cutoff{\ell} \leq 1\) with 
    \begin{equation}
        \label{dslp:eq:def:Cutoff}
        \cutoff{\ell} = \begin{cases*}
            0, \quad \text{in } \Patch{\ell-1}{\omega} \\
            1, \quad \text{in } \fulldomain \setminus \Patch{\ell}{\omega}
        \end{cases*}
        , \qquad \text{ and } \qquad \norm{\grad \cutoff{\ell}}_{L^{\infty}(\fulldomain)} \leq \CCutoff h^{-1} \;.
    \end{equation}
	Then, using this cutoff function \(\cutoff{\ell}\) we derive with the product rule 
	\begin{equation}
		\label{dslp:eq:CNlocalData:terms}
		\begin{aligned}
		\tnorm[\fulldomain \setminus \Patch{\ell}{\omega}]{\uCN{n}}^2 
		=& \; \mathcal{K}_{\fulldomain \setminus \Patch{\ell}{\omega}}(\uCN{n}, \uCN{n}) \\
		\leq& \;  \ip[\fulldomain]{ A \grad \uCN{n}, \cutoff{\ell}\grad \uCN{n}} + \frac{4}{\tss}\bbil[\fulldomain]{\uCN{n}, \cutoff{\ell}\uCN{n}}  \\
		=& \; \mathcal{K}_{\fulldomain}(\uCN{n}, \NodalInt (\cutoff{\ell}\uCN{n}))
			+\mathcal{K}_{\fulldomain}(\uCN{n},(\id - \NodalInt) (\cutoff{\ell}\uCN{n})) -  \ip[\fulldomain]{A \grad \uCN{n}, \uCN{n} \grad \cutoff{\ell}}.
		\end{aligned}
	\end{equation}
	We now treat each of the three occurring terms individually. 

	For the first term, we reformulate the CN scheme \eqref{dslp:eq:CNfirstOrderFormulation} to
	\begin{align*}
		\tss[4]\mathcal{K}_{\fulldomain}(\uCN{n}, w_h) 
		=& \; \bbil[\fulldomain]{\uCN{n-1}, w_h} - \tss[4] \abil[\fulldomain]{\uCN{n-1}, w_h} 
		+ \tau \bbil[\fulldomain]{\vCN{n-1}, w_h} + \tss[2] \bbil[\fulldomain]{\overline{f}_h^n, w_h},
	\end{align*}
	for all \(w_h \in \PkZeroTh[]\).
	The choice \(w_h = \NodalInt(\cutoff{\ell} \uCN{n}) \in \PkZeroTh[]\) is valid, but since \(\supp \NodalInt(\cutoff{\ell} \uCN{n}) \subseteq \fulldomain \setminus \Patch{\ell-1}{\omega}\) and by the assumption on the support of \(\uCN{n-1}, \vCN{n-1}, f_h^n\) and \(f_h^{n-1}\), we have 
	\begin{equation}
		\label{dslp:eq:CNlocalData:term1}
		\mathcal{K}_{\fulldomain}(\uCN{n}, \NodalInt (\cutoff{\ell}\uCN{n})) = 0 .
	\end{equation}

	Similarly, since 
	\begin{equation*}
		\supp (\id - \NodalInt)(\cutoff{\ell} \uCN{n}) \subseteq \bigcup_{K \in \bLayer{\ell}} K
	\end{equation*}
	also the second term from \eqref{dslp:eq:CNlocalData:terms} vanishes on every element \(K \in \Th \setminus \bLayer{\ell}\). 
	
	On \(K \in \bLayer{\ell}\), the Cauchy--Schwarz inequality, Lemma~\ref{dslp:Lem:StabilityIhProductIh}, the product rule, and \eqref{dslp:eq:def:Cutoff} yield
	\begin{align*}
		\mathcal{K}_{K}(\uCN{n},(\id - \NodalInt)(\cutoff{\ell} \uCN{n}))
		&\leq \Cint \beta^{\half}\Bigl( \norm[a,K]{\uCN{n}} \abs[1,K]{\cutoff{\ell}\uCN{n}} + \frac{4}{\tss} \norm[b,K]{\uCN{n}}\norm[K]{\cutoff{\ell}\uCN{n}}\Bigr)\\
		&\leq \Cint \frac{\beta^{\half}}{\alpha^{\half}} \Bigl( \frac{\CCutoff}{h} \norm[a,K]{\uCN{n}} \norm[b,K]{\uCN{n}} + \norm[a,K]{\uCN{n}}^2 + \frac{4}{\tss} \norm[b,K]{\uCN{n}}^2\Bigr).
	\end{align*} 
	We estimate the mixed term further with a weighted Young inequality, i.e., 
	\begin{equation}
        \label{dslp:eq:estimateMixedNormCutoffTerm}
		\norm[b,K]{\uCN{n}} \norm[a,K]{\uCN{n}} \leq \ts[4] \norm[a,K]{\uCN{n}}^2 + \frac{1}{\ts}\norm[b,K]{\uCN{n}}^2 = \ts[4] \tnorm[K]{\uCN{n}}^2,
	\end{equation}
	which together with the previous inequality combined on all \(K \in \bLayer{\ell}\) leads to
	\begin{equation}
		\label{dslp:eq:CNlocalData:term2}
		\mathcal{K}_{\fulldomain}(\uCN{n},(\id - \NodalInt)(\cutoff{\ell} \uCN{n})) \leq \Cint \frac{\beta^{\half}}{\alpha^{\half}} \Bigl(\frac{\CCutoff}{4}\frac{\tau}{h} + 1\Bigr)  \tnorm[\bLayer{\ell}]{\uCN{n}}^2.
	\end{equation}

    Lastly, since also
	\(\supp(\grad\cutoff{\ell}) \subseteq \bigcup_{K \in \bLayer{\ell}} K \),
	the last term from \eqref{dslp:eq:CNlocalData:terms} is estimated directly on the transition layer, using again \eqref{dslp:eq:estimateMixedNormCutoffTerm}, from which we obtain
	\begin{equation}
        \label{dslp:eq:CNlocalData:term3}
		\abs{\ip[\fulldomain]{A \grad \uCN{n}, \uCN{n} \grad \cutoff{\ell}}}
		\leq
		\frac{\beta^{\half}}{\alpha^{\half}} \frac{\CCutoff}{h}
		\sum_{K \in \bLayer{\ell}}\norm[a,K]{\uCN{n}}\norm[b,K]{\uCN{n}}
		\leq
        \frac{\beta^{\half}}{\alpha^{\half}} \frac{\CCutoff}{4} \frac{\tau}{h} 
		\tnorm[\bLayer{\ell}][\big]{\uCN{n}}^2.
	\end{equation}

	Applying \eqref{dslp:eq:CNlocalData:term1}, \eqref{dslp:eq:CNlocalData:term2}, \eqref{dslp:eq:CNlocalData:term3} to \eqref{dslp:eq:CNlocalData:terms} yields with \(\Cint > 1\)
    \begin{equation}
        \label{dslp:eq:defWtauhmin}
        \tnorm[\fulldomain \setminus \Patch{\ell}{\omega}]{\uCN{n}}^2 \leq 
        \Cint \frac{\beta^{\half}}{\alpha^{\half}} \Bigl(\frac{\CCutoff}{2}\frac{\tau}{h}  + 1\Bigr) 
        \tnorm[\bLayer{\ell}]{\uCN{n}}^2 
        \eqcolon W_{\tau/h} \tnorm[\bLayer{\ell}]{\uCN{n}}^2 .
    \end{equation}
    Using that
	\begin{equation*}
		\bLayer{\ell} = \Patch{\ell}{\omega} \setminus \Patch{\ell -1}{\omega} = \bigl(\Th[\fulldomain] \setminus \Patch{\ell -1}{\omega}\bigr)\setminus \bigl(\Th[\fulldomain] \setminus \Patch{\ell}{\omega}\bigr) ,
	\end{equation*}
    the inequality~\eqref{dslp:eq:defWtauhmin} can be rewritten to
    \begin{equation*}
        \tnorm[\fulldomain \setminus \Patch{\ell}{\omega}][\big]{\uCN{n}}^2 \leq  W_{\tau/h}  \Bigl( \tnorm[\fulldomain \setminus \Patch{\ell -1 }{\omega}][\big]{\uCN{n}}^2 -\tnorm[\fulldomain \setminus \Patch{\ell}{\omega}][\big]{\uCN{n}}^2 \Bigr),
    \end{equation*}
    which is equivalent to 
    \begin{equation*}
        \tnorm[\fulldomain \setminus \Patch{\ell}{\omega}][\big]{\uCN{n}}^2  \leq \gamma^2 \tnorm[\fulldomain \setminus \Patch{\ell-1}{\omega}][\big]{\uCN{n}}^2, 
	\end{equation*}
	with \(\gamma^2 = W_{\tau/h} / (1 + W_{\tau/h})\). 
    Iterating this estimate yields
	\begin{equation*}
		\tnorm[\fulldomain \setminus \Patch{\ell}{\omega}]{\uCN{n}}
		\leq
		\gamma^\ell
		\tnorm[\fulldomain \setminus \Patch{0}{\omega}]{\uCN{n}}.
	\end{equation*}
	Since \(\omega \subset \dom\Patch{0}{\omega}\), the right-hand side can be enlarged to
	\begin{equation*}
		\tnorm[\fulldomain \setminus \omega]{\uCN{n}}.
	\end{equation*}
	The final estimate follows from Lemma~\ref{dslp:lem:CN-NormChangeAbsentData}.
\end{proof}

Next, we prove a localization theorem in the spirit of \cite[Theorem 3.2]{GalM23} but restricted to one time step only. 
Here, we investigate the difference between a global CN approximation and one that is computed only on the domain associated with the patch \(\Patch{\ell}{\omega}\) around the area \(\omega\), where the data is supported.
Let~\(\solCNloc{n} = \pmatT{\uCNloc{n}\;\vCNloc{n}}\) be the CN approximation based on the given data on \(\omega\), but computed locally on \(\dom \Patch{\ell}{\omega}\) with homogeneous boundary conditions on the artificial boundary faces. 
We then extend this approximation by zero to \(\fulldomain\) to compare it to the global CN approximation \(\solCN{n}\).
We have the following localization result.

\begin{lemma}[CN localization error]
	\label{dslp:Lem:localizedCN}
	Assume \(\uCN{n-1}, \vCN{n-1}, f_h^n, f_h^{n-1}\) to be supported in \(\omega \subset \fulldomain\).  Then, for \(\ell \in \mathbb{N}\), 
	\begin{equation}
        \label{dslp:Lem:localizedCN:equation}
		\begin{aligned}
			\tnorm[\fulldomain][\big]{\uCN{n} - \uCNloc{n}} 
			&\leq C \max \monosetc*{\frac{\tau}{h},1}\gamma^{\ell} \tnorm[\fulldomain \setminus \omega][\big]{\uCN{n}} \\
			&\leq C \max \monosetc*{\frac{\tau}{h},1}\gamma^{\ell} \norm[X_h(\fulldomain \setminus \omega)][\big]{\solCN{n}},
		\end{aligned}
	\end{equation}
	with \(\gamma\) from Lemma~\ref{dslp:lem:CNlocalizedData} and a constant \(C>0\), which only depends on the material bounds \(\alpha, \beta\), the polynomial degree \(k\), the spatial dimension, and the shape regularity of \(\Th\).
\end{lemma}

\begin{proof}
    For a test function \(w_h \in \PkZeroTh[]\) with \(\supp w_h \subseteq \Patch{\ell}{\omega}\) we can subtract the weak formulations of \(\solCN{n}\) and \(\solCNloc{n}\), which yields 
	\begin{equation}
		\begin{aligned}
		\bbil[\fulldomain]{\uCN{n}- \uCNloc{n}, w_h} 
		&=\ts[2] \bbil[\fulldomain]{\vCN{n}- \vCNloc{n}, w_h} ,\\
		\bbil[\fulldomain]{\vCN{n}- \vCNloc{n}, w_h} 
		&= - \ts[2] \abil[\fulldomain]{\uCN{n} - \uCNloc{n}, w_h}.
		\end{aligned}
	\end{equation}
	By combining both lines we get 
	\begin{equation}
		\label{dslp:eq:localizedDifferenceWeakForm}
		\bbil[\fulldomain]{\uCN{n}- \uCNloc{n}, w_h} + \tss[4] \abil[\fulldomain]{\uCN{n} - \uCNloc{n}, w_h} = 0 .
	\end{equation}

	Thus, since \(\supp \uCNloc{n} \subseteq \Patch{\ell}{\omega}\), we can also write 
	\begin{align*} 
		\tnorm[\fulldomain][\big]{\uCN{n} -\uCNloc{n}}^2 
		&= \abil[\fulldomain]{\uCN{n} - \uCNloc{n}, \uCN{n} - \uCNloc{n}} + \mfrac{4}{\tss} \bbil[\fulldomain]{\uCN{n}- \uCNloc{n}, \uCN{n}- \uCNloc{n}} \\
		&= \abil[\fulldomain]{\uCN{n} - \uCNloc{n}, \uCN{n} - w_h} + \mfrac{4}{\tss} \bbil[\fulldomain]{\uCN{n}- \uCNloc{n}, \uCN{n}- w_h} .
	\end{align*}

    Next, we choose \(w_h = \NodalInt((1- \cutoff{\ell})\uCN{n})\) with \(\cutoff{\ell}\) as defined in \eqref{dslp:eq:def:Cutoff}, such that indeed \(\supp w_h \subseteq \Patch{\ell}{\omega} \).
	Then, by the discrete Cauchy--Schwarz inequality and the identity \(\uCN{n} - \NodalInt((1- \cutoff{\ell})\uCN{n}) = 0\) on \(\Patch{\ell -1}{\omega}\), we have 
    \begin{align}
		\label{dslp:eq:localizedDiff:CSI}
		\nonumber \tnorm[\fulldomain][\big]{\uCN{n} - \uCNloc{n}}^2 
		\leq& \; \norm[a, \fulldomain][\big]{\uCN{n} - \uCNloc{n}} 
		\norm[a, \fulldomain \setminus \Patch{\ell-1}{\omega}][\big]{\uCN{n}- \NodalInt((1- \cutoff{\ell})\uCN{n})} \\
		\nonumber &+ \mfrac{2}{\ts} \norm[b,\fulldomain][\big]{\uCN{n} - \uCNloc{n}} \,
		\mfrac{2}{\ts} \norm[b, \fulldomain \setminus \Patch{\ell-1}{\omega}][\big]{\uCN{n}- \NodalInt((1- \cutoff{\ell})\uCN{n})} \\
		\leq& \; \tnorm[\fulldomain][\big]{\uCN{n} - \uCNloc{n}}
		\tnorm[\fulldomain \setminus \Patch{\ell-1}{\omega}][\big]{\uCN{n}- \NodalInt((1- \cutoff{\ell})\uCN{n})}.
	\end{align}
    We can simplify the second factor of \eqref{dslp:eq:localizedDiff:CSI} and get
	\begin{align*}
		\tnorm[\fulldomain \setminus \Patch{\ell-1}{\omega}][\big]{\uCN{n}- \NodalInt((1- \cutoff{\ell})\uCN{n})} = \tnorm[\fulldomain \setminus \Patch{\ell-1}{\omega}][\big]{\NodalInt(\cutoff{\ell}\uCN{n})}.
	\end{align*}
    We again apply Lemma~\ref{dslp:Lem:StabilityIhProductIh} to obtain 
	\begin{equation*}
		\norm[\fulldomain \setminus \Patch{\ell -1}{\omega}]{\NodalInt(\cutoff{\ell}\uCN{n})} \leq \Cint \frac{\beta^{\half}}{\alpha^{\half}} \norm[\fulldomain \setminus \Patch{\ell -1}{\omega}]{\uCN{n}},
	\end{equation*}
	and by the product rule and \eqref{dslp:eq:def:Cutoff} also 
	\begin{equation*}
		\norm[a,\fulldomain \setminus \Patch{\ell -1}{\omega}]{\NodalInt(\cutoff{\ell}\uCN{n})}
		\leq \Cint \frac{\beta^{\half}}{\alpha^{\half}} \Bigl( \frac{\CCutoff}{h} \norm[\fulldomain \setminus \Patch{\ell - 1}{\omega}]{\uCN{n}} +  \norm[a,\fulldomain \setminus \Patch{\ell -1}{\omega}]{\uCN{n}} \Bigr).
	\end{equation*} 
	Together, these bounds yield
	\begin{equation}
		\label{dslp:eq:boundTnormCutoffFunction}
			\tnorm[\fulldomain \setminus \Patch{\ell-1}{\omega}][\big]{\NodalInt(\cutoff{\ell}\uCN{n})} \leq 
		\Cint \frac{\beta^{\half}}{\alpha^{\half}} \Bigl( 1+ \frac{\CCutoff}{2} \frac{\tau}{h}\Bigr) \tnorm[\fulldomain \setminus \Patch{\ell -1}{\omega}]{\uCN{n}}.
	\end{equation}
	We can thus combine \eqref{dslp:eq:boundTnormCutoffFunction} with \eqref{dslp:eq:localizedDiff:CSI} and the definition of \(W_{\tau/h}\) from~\eqref{dslp:eq:defWtauhmin} to obtain 
	\begin{equation*}
		\begin{aligned}
			\tnorm[\fulldomain][\big]{\uCN{n} - \uCNloc{n}} 
			\leq W_{\tau/h} \tnorm[\fulldomain \setminus \Patch{\ell -1}{\omega}]{\uCN{n}} = W_{\tau/h}^{\half} ( 1+ W_{\tau/h})^{\half} \gamma \tnorm[\fulldomain \setminus \Patch{\ell -1}{\omega}]{\uCN{n}} 
		\end{aligned}
	\end{equation*}
	with \(\gamma = W_{\tau/h}^{\half} ( 1+ W_{\tau/h})^{-\half} \) as in Lemma~\ref{dslp:lem:CNlocalizedData}.
	In the case \(\tau \geq h\), which is mostly of interest in this work\footnote{All results remain valid also in the case \(\tau<h\), where the factor \(\tau/h\) can be bounded by \(1\). This regime is less relevant for the present discussion, since an explicit CFL restriction would typically already be satisfied.}, we can further bound 
	\begin{equation*}
		\begin{aligned}
			\tnorm[\fulldomain][\big]{\uCN{n} - \uCNloc{n}} 
			&\leq W_{\tau/h}^{\half} ( 1+ W_{\tau/h})^{\half} \gamma \tnorm[\fulldomain \setminus \Patch{\ell -1}{\omega}]{\uCN{n}} \\
			&\leq C \frac{\tau}{h} \gamma  \tnorm[\fulldomain \setminus \Patch{\ell -1}{\omega}]{\uCN{n}},
		\end{aligned}
	\end{equation*}
	with a constant \(C> 0\) which only depends on the polynomial degree \(k\), the spatial dimension, the material bounds \(\alpha, \beta\) and the shape regularity of \(\Th\).  
	If \(\tau < h\) we just bound \(\tau / h\) by \(1\).
	Together with Lemma~\ref{dslp:lem:CNlocalizedData} and Lemma~\ref{dslp:lem:CN-NormChangeAbsentData} this yields the claim.
\end{proof}

Lemma~\ref{dslp:lem:CNlocalizedData} and Lemma~\ref{dslp:Lem:localizedCN} provide the two main mechanisms used below to bound the prediction error.
First, if we are separated from the support of the data, as in Lemma~\ref{dslp:lem:CNlocalizedData}, we obtain a layer-wise contraction until we reach the data.
Second, if we compare a globally computed CN approximation with one computed only on a patch around the support of the data, we obtain decay with respect to the number of layers in this patch, as in Lemma~\ref{dslp:Lem:localizedCN}.

Both mechanisms enter the prediction error bound in Lemma~\ref{dslp:lem:TestLocalizedPredictionEstimateNew}, which we prove in Section~\ref{dslp:subsec:localizedPrediction:EstimatesPrediction}.
Before proving this result, we transfer these findings to the CN increments underlying the localized prediction from Section~\ref{dslp:sec:method}.

\subsection{Variationally localized increments}
\label{dslp:subsec:localizedPrediction:VariationalLocalization}

Inspired by \cite{GalM23}, Lemma~\ref{dslp:lem:CNlocalizedData} and Lemma~\ref{dslp:Lem:localizedCN} established two localization mechanisms, adjusted to the first-order formulation of a standard CN step with localized data.
We keep these results separate because they isolate the basic decay arguments in the simplest setting. 

For the localized prediction used in the domain splitting method, however, the data are not localized directly.
Instead, the prediction is formulated for the CN increment, and localization is introduced through the right-hand side by applying a cutoff to the test function.
This new localization technique is crucial because it preserves the variational structure of the increment problem and localizes only the corresponding right-hand side. 
This allows us to estimate the prediction error in the discrete energy norm \(\norm[X_h]{\cdot}\). 
A direct localization of the data would instead lead to estimates in the \(\tnorm{\cdot}\)-norm, and the required conversion to the \(X_h\)-norm would introduce a factor of order~\(\tau^{-1}\).

Within this section, we transfer the preceding localization mechanisms to this increment formulation. 
This is the setting needed for the prediction step in Section~\ref{dslp:sec:method}: the cutoff selects the part of the increment right-hand side associated with one prediction interface, and the resulting localized increment is then computed only on a patch around that interface.

Throughout this section, let \(\Lambda \in \PkTh\) satisfy \(0 \leq \Lambda \leq 1\) and \(\norm[L^{\infty}(\fulldomain)]{\grad \Lambda} \leq C_{\Lambda} h^{-1}\) with a constant~\(C_{\Lambda}\) that only depends on the shape regularity of the mesh. 
Recall the increment formulation of the CN method from \eqref{dslp:eq:CN-increment-formulation}. We first define \(\delta u_\Lambda^n \in \PkZeroTh\) by
\begin{equation}
	\label{dslp:eq:TestLocalizedIncrement}
	\frac{\tau^2}{4}\mathcal{K}_{\fulldomain}(\delta u_\Lambda^n,w_h)
	=
	\tau\bbil[\fulldomain]{\vCN{n-1},\NodalInt(\Lambda w_h)}
	-\frac{\tau^2}{2}\abil[\fulldomain]{\uCN{n-1},\NodalInt(\Lambda w_h)}
	+\frac{\tau^2}{2}\bbil[\fulldomain]{\overline{f}_h^n,\NodalInt(\Lambda w_h)}
\end{equation}
for all \(w_h\in \PkZeroTh\) with \(\mathcal{K}_{\fulldomain}(\cdot,\cdot)\) as defined in \eqref{dslp:eq:mathcalK-Def}. Note that \(\delta u_\Lambda^n\) is still computed on \(\fulldomain\). 

Thus, we also define a locally computed increment on 
\begin{equation*}
    D_\Lambda^\ell \coloneqq \dom \; \Patch{\ell}{\supp \Lambda} =  \interior \bigcup_{K \in \Patch{\ell}{\supp \Lambda}} K,
\end{equation*}
which is denoted by~\(\widehat{\delta}u^n_{\Lambda}\in \PkZeroTh[D_\Lambda^\ell] \) and defined by
\begin{equation}
	\label{dslp:eq:TestLocalizedIncrementLocalCompute}
	\begin{aligned}
	\frac{\tau^2}{4} \mathcal{K}_{D_\Lambda^\ell}\!(\widehat{\delta} u_\Lambda^n,w_h)
	=  
	\tau\bbil[D_\Lambda^\ell]{\vCN{n-1},\NodalInt(\Lambda w_h)}
	-\frac{\tau^2}{2}\abil[D_\Lambda^\ell]{\uCN{n-1},\NodalInt(\Lambda w_h)} 
	+\frac{\tau^2}{2}\bbil[D_\Lambda^\ell]{\overline{f}_h^n,\NodalInt(\Lambda w_h)}
	\end{aligned}
\end{equation}
for all test functions \(w_h \in \PkZeroTh[D_\Lambda^\ell]\). 
The goal of this subsection is to derive an estimate for~\(\tnorm[\fulldomain]{\delta u_\Lambda^n-\widehat{\delta}u^n_{\Lambda}}\).
To start, we state a stability result for \(\delta u_\Lambda^n\).

\begin{lemma}[Stability of localized increments]
	\label{dslp:lem:TestLocalizedIncrementStability}
	For \(\delta u_\Lambda^n\) from \eqref{dslp:eq:TestLocalizedIncrement}, we have
	\begin{equation*}
		\tnorm[\fulldomain]{\delta u_\Lambda^n}
		\leq
			C \Bigl(1 + \frac{\tau}{h}\Bigr)
			\left(
				\norm[X_h(\fulldomain)][\Big]{\pmat{\uCN{n-1}\\\vCN{n-1}}}
				+\tau\norm[b, \fulldomain][\big]{\overline{f}_h^n}
			\right),
	\end{equation*}
	with a constant \(C>0\) that only depends on the material bounds \(\alpha, \beta,\) the polynomial degree \(k\), the spatial dimension, and the shape regularity of the mesh.
\end{lemma}

\begin{proof}
    Lemma~\ref{dslp:Lem:StabilityIhProductIh}, the product rule, and the bound on \(\grad\Lambda\) yield
	\begin{equation}
		\label{dslp:eq:TestLocalizedProductBound}
        \begin{aligned}
            \norm[b,\fulldomain]{\NodalInt(\Lambda w_h)}
		    &\leq  \Cint\frac{\beta^{\half}}{\alpha^{\half}} \norm[b, \supp \Lambda]{w_h} \leq  \Cint\frac{\beta^{\half}}{\alpha^{\half}} \ts[2] \tnorm[\supp \Lambda]{w_h}, \\
            \norm[a, \fulldomain]{\NodalInt(\Lambda w_h)}
            &\leq \Cint \frac{\beta^{\half}}{\alpha^{\half}}   \Bigl(
            \norm[a, \supp \Lambda]{w_h}+\frac{C_{\Lambda}}{h} \norm[b,\supp \Lambda]{w_h} \Bigr) \\
            &\leq \Cint \frac{\beta^{\half}}{\alpha^{\half}} \Bigl( 1+ \frac{C_{\Lambda}}{2} \ts[h]\Bigr) \tnorm[\supp \Lambda]{w_h}.
        \end{aligned}
	\end{equation}
    Testing \eqref{dslp:eq:TestLocalizedIncrement} with \(w_h = \delta u_\Lambda^n \), using the Cauchy--Schwarz inequality, and \eqref{dslp:eq:TestLocalizedProductBound}, we obtain
    \begin{align*}
        \tss[4] \tnorm[\fulldomain]{\delta u_\Lambda^n}^2
        &= \frac{\tau^2}{4}\mathcal{K}_{\fulldomain}(\delta u_\Lambda^n,\delta u_\Lambda^n) \\
        &\leq \Cint \frac{\beta^{\half}}{\alpha^{\half}} \Bigl( \tss[2] \norm[b,\fulldomain]{\vCN{n-1}} + \tss[2] \Bigl( 1+ \frac{C_{\Lambda}}{2} \ts[h]\Bigr) \norm[a,\fulldomain]{\uCN{n-1}} + \frac{\tau^3}{4} \norm[b,\fulldomain]{\overline{f}_h^n}  \Bigr)\tnorm[\fulldomain]{\delta u_\Lambda^n}.
    \end{align*}
    Dividing by \(\tss[4] \tnorm[\fulldomain]{\delta u_\Lambda^n}\) and rearranging the constants yields the result.
\end{proof}

Next, we show a decay result similar to Lemma~\ref{dslp:lem:CNlocalizedData} for the global CN increment~\(\delta u_\Lambda^n\).
\begin{lemma}[Decay of localized increments]
	\label{dslp:lem:TestLocalizedIncrementDecay}
	Let \(\delta u_\Lambda^n\) be defined by \eqref{dslp:eq:TestLocalizedIncrement}.  Then, for \(\ell \in \mathbb{N}_0\),
	\begin{equation*}
		\tnorm[\fulldomain\setminus D_\Lambda^\ell]{\delta u_\Lambda^n}
		\leq
			C \Bigl(1 + \frac{\tau}{h}\Bigr)\gamma^\ell
			\left(
				\norm[X_h(\fulldomain)][\Big]{\pmat{\uCN{n-1}\\\vCN{n-1}}}
				+\tau\norm[b,\fulldomain][\big]{\overline{f}_h^n}
			\right),
	\end{equation*}
	with the same \(\gamma\) as in Lemma~\ref{dslp:lem:CNlocalizedData}. The constant \(C\) depends only on the material bounds \(\alpha, \beta \), the polynomial degree \(k\), the spatial dimension, and the shape regularity of the mesh.
\end{lemma}

\begin{proof}
    The proof is very similar to the one from Lemma~\ref{dslp:lem:CNlocalizedData}. The case \(\ell = 0\) holds trivially. 
    For~\(\ell \geq 1\), we introduce a cutoff \(\cutoff{\ell}\) as in
    \eqref{dslp:eq:def:Cutoff}, now with \(\omega=\supp\Lambda\).  
    
    Thus,
    \(\cutoff{\ell}=0\) on~\(K\in\Patch{\ell-1}{\supp\Lambda}\),
    \(\cutoff{\ell}=1\) on
    \(K \in \Th\setminus \Patch{\ell}{\supp \Lambda}\), and~\(\norm[L^\infty]{\grad\cutoff{\ell}}\leq \CCutoff h^{-1}\).
    With \(z_h=\delta u_\Lambda^n\), we insert
    \(\NodalInt(\cutoff{\ell}z_h)\) and use
    \(\grad(\cutoff{\ell}z_h)=\cutoff{\ell}\grad z_h+z_h\grad\cutoff{\ell}\)
    to get
    \begin{equation}
        \label{dslp:eq:IncrementCNlocalData:terms}
            \tnorm[\fulldomain \setminus D_\Lambda^\ell]{z_h}^2
            \leq  
            \mathcal{K}_{\fulldomain}\bigl(z_h, \NodalInt (\cutoff{\ell} z_h)\bigr)
            +  \mathcal{K}_{\fulldomain}\bigl(z_h,\bigl(\id - \NodalInt\bigr)  (\cutoff{\ell} z_h)\bigr) 
            - \ip[\fulldomain]{A \grad z_h, z_h \grad \cutoff{\ell}}.
    \end{equation}

    The first term vanishes because, for \(w_h=\NodalInt(\cutoff{\ell}z_h)\), the localization removes the right-hand side of \eqref{dslp:eq:TestLocalizedIncrement}.  
    Indeed, all Lagrange nodal values of \(\NodalInt(\Lambda w_h)\) vanish, since
    \begin{equation*}
        \NodalInt(\Lambda w_h)(a_j)
        =
        \Lambda(a_j)\cutoff{\ell}(a_j)z_h(a_j).
    \end{equation*}
    If \(\Lambda(a_j)\neq0\), then
    \(a_j\in\supp\Lambda\subseteq \overline{D}_\Lambda^\ell\) and
    \(\cutoff{\ell}(a_j)=0\). Otherwise, the factor \(\Lambda(a_j)\) is
    already zero.  Hence, \(\NodalInt(\Lambda w_h)=0\), and
    \eqref{dslp:eq:TestLocalizedIncrement} gives
    \begin{equation}
        \label{dslp:eq:IncrementCNlocalData:term1}
        \mathcal{K}_{\fulldomain}
        \bigl(z_h,\NodalInt(\cutoff{\ell}z_h)\bigr)=0 .
    \end{equation}

    Similarly, the second term is only non-zero on cells contained in \(\bLayer{\ell} = \Patch{\ell}{\supp \Lambda} \setminus \Patch{\ell-1}{\supp \Lambda}\). By the same derivation as for \eqref{dslp:eq:CNlocalData:term2}, we obtain 
    \begin{equation}
        \label{dslp:eq:IncrementCNlocalData:term2}
        \mathcal{K}_{\fulldomain}(z_h,(\id-\NodalInt)(\cutoff{\ell} z_h))
        \leq \Cint \frac{\beta^{\half}}{\alpha^{\half}} \Bigl(\frac{\CCutoff}{4}\frac{\tau}{h}  + 1\Bigr) \tnorm[\bLayer{\ell}]{z_h}^2.
    \end{equation}
    Following the steps from \eqref{dslp:eq:CNlocalData:term3}, we finally get for the third term
    \begin{equation}
        \label{dslp:eq:IncrementCNlocalData:term3}
        \abs{\ip[\fulldomain]{A\grad z_h,z_h\grad\cutoff{\ell}}}
        \leq 
        \frac{\beta^{\half}}{\alpha^{\half}} \frac{\CCutoff}{4} \frac{\tau}{h} 
        \tnorm[\bLayer{\ell}][\big]{z_h}^2.
    \end{equation}
    Substituting \eqref{dslp:eq:IncrementCNlocalData:term1}, \eqref{dslp:eq:IncrementCNlocalData:term2}, and \eqref{dslp:eq:IncrementCNlocalData:term3} into \eqref{dslp:eq:IncrementCNlocalData:terms} yields
    \begin{equation*}
        \tnorm[\fulldomain \setminus D_{\Lambda}^{\ell}][\big]{z_h}^2 \leq  W_{\tau/h}  \Bigl( \tnorm[\fulldomain \setminus D_{\Lambda}^{\ell-1}][\big]{z_h}^2 -\tnorm[\fulldomain \setminus D_{\Lambda}^{\ell}][\big]{z_h}^2 \Bigr),
    \end{equation*}
    from which the claim follows as in Lemma~\ref{dslp:lem:CNlocalizedData} with a final application of Lemma~\ref{dslp:lem:TestLocalizedIncrementStability}. 
\end{proof}

Our next goal is to also recover an estimate similar to Lemma~\ref{dslp:Lem:localizedCN} for the difference of the global increment~\(\delta u_\Lambda^n\) and the locally computed increment~\(\widehat{\delta}u^n_{\Lambda}\) as defined in \eqref{dslp:eq:TestLocalizedIncrementLocalCompute}.

\begin{lemma}[Truncation error of localized increments]
	\label{dslp:lem:TestLocalizedIncrementTruncation}
	Let \(\widehat{\delta} u_{\Lambda}^n\) be the local increment on~\(D_\Lambda^\ell\), defined in~\eqref{dslp:eq:TestLocalizedIncrementLocalCompute}, with~\(\ell\in \mathbb{N}\).  Then,
	\begin{equation*}
		\tnorm[\fulldomain][\big]{\delta u_\Lambda^n-\widehat{\delta}u^n_{\Lambda}}
		\leq
			C \Bigl(1 + \frac{\tau}{h}\Bigr) \max \monosetc*{\frac{\tau}{h},1} \gamma^{\ell}
			\left(
				\norm[X_h(\fulldomain)][\Big]{\pmat{\uCN{n-1}\\\vCN{n-1}}}
				+\tau\norm[b,\fulldomain][\big]{\overline{f}_h^n}
			\right),
	\end{equation*}
    with \(\gamma\) from Lemma~\ref{dslp:lem:CNlocalizedData} and a constant \(C\) depending on the material bounds \(\alpha, \beta \), the polynomial degree \(k\), the spatial dimension, and the shape regularity of the mesh.
\end{lemma}

\begin{proof}
    We extend \(\widehat{\delta}u_\Lambda^n\) by zero outside \(D_\Lambda^\ell\), and set \(e_h=\delta u_\Lambda^n-\widehat{\delta}u_\Lambda^n\).
	For every~\(w_h\in\PkZeroTh[D_\Lambda^\ell]\), the global problem
	\eqref{dslp:eq:TestLocalizedIncrement} and the local problem
	\eqref{dslp:eq:TestLocalizedIncrementLocalCompute} have the same right-hand
	side, because it holds that \(\supp\NodalInt(\Lambda w_h)\subseteq \Patch{\ell}{\supp\Lambda}\).  Hence,
	\begin{equation*}  
		\mathcal K_{\fulldomain}(e_h,w_h)=0
		\qquad
		\text{for all }w_h\in\PkZeroTh[D_\Lambda^\ell].
    \end{equation*}
    We now proceed similarly to the proof of Lemma~\ref{dslp:Lem:localizedCN} and choose \(w_h = \NodalInt((1- \cutoff{\ell})\delta u_\Lambda^n)\) with \(\cutoff{\ell}\) as defined in \eqref{dslp:eq:def:Cutoff}. 
    Then \(w_h\in\PkZeroTh[D_\Lambda^\ell]\), and therefore
	\begin{equation*}
		\tnorm[\fulldomain]{e_h}^2
		=
		\mathcal K_{\fulldomain}(e_h,e_h)
		=
		\mathcal K_{\fulldomain}
		\bigl(e_h,\delta u_\Lambda^n-w_h\bigr).
    \end{equation*}
    The difference in the second argument is supported only on 
	\(K \in \Th\setminus\Patch{\ell-1}{\supp\Lambda}\), because
	\begin{equation*}
		\delta u_\Lambda^n-\NodalInt((1-\cutoff{\ell})\delta u_\Lambda^n)
		=
		\NodalInt(\cutoff{\ell}\delta u_\Lambda^n).   
    \end{equation*}
    Using the Cauchy--Schwarz inequality in the \(\tnorm{\cdot}\)-associated inner product gives
	\begin{equation}
        \label{dslp:eq:Tnorm-eh}
       \tnorm[\fulldomain]{e_h}
		\leq
		\tnorm[\fulldomain\setminus D_\Lambda^{\ell-1}]
		{\NodalInt(\cutoff{\ell}\delta u_\Lambda^n)},
    \end{equation}
    with \(D_\Lambda^{\ell-1} = \dom \Patch{\ell-1}{\supp\Lambda}\).
    As for \eqref{dslp:eq:boundTnormCutoffFunction}, we derive 
    \begin{equation*}
        \tnorm[\fulldomain \setminus D_\Lambda^{\ell-1}][\big]{\NodalInt(\cutoff{\ell}\delta u_\Lambda^n)} \leq 
		\Cint \frac{\beta^{\half}}{\alpha^{\half}} ( 1+ \frac{\CCutoff}{2} \frac{\tau}{h}) \tnorm[\fulldomain \setminus D_\Lambda^{\ell-1}]{\delta u_\Lambda^n},
    \end{equation*}
    from which it follows with \eqref{dslp:eq:Tnorm-eh} as in Lemma~\ref{dslp:Lem:localizedCN} that
    \begin{equation*}
        \tnorm[\fulldomain]{e_h}
        \leq C \max\monosetc*{\frac{\tau}{h},1}\gamma \tnorm[\fulldomain \setminus D_\Lambda^{\ell-1}]{\delta u_\Lambda^n},
    \end{equation*}
    with a constant \(C>0\) depending only on \(\alpha,\beta,k\), the spatial dimension, and the shape regularity of \(\Th\). The application of Lemma~\ref{dslp:lem:TestLocalizedIncrementDecay} then yields the result.
\end{proof}

\subsection{Estimate for the localized implicit prediction}
\label{dslp:subsec:localizedPrediction:EstimatesPrediction}
After the preparatory work from Sections~\ref{dslp:subsec:localizedPrediction:CNlocalized} and~\ref{dslp:subsec:localizedPrediction:VariationalLocalization}, we are now in a position to derive the prediction bound stated in Lemma~\ref{dslp:lem:TestLocalizedPredictionEstimateNew}, which is used within the proof of the convergence result in Theorem~\ref{dslp:thm:ConvergenceOverlapCondition}.
The last missing step is taking the boundary lifting from Assumption~\ref{dslp:ass:boundaryLifting} into account. 
The idea behind that is natural: The prediction \(\uLP{i}{n}\) does not have to be accurate on the
full overlap \(\SDov{i}\).  It is sufficient to control it on the lifting strip
\(\LiftStrip{i}\), because \eqref{dslp:eq:BoundaryLiftingStability}
turns such a strip estimate into a bound for the local difference.  The cutoff
\(\chi_i\) is therefore chosen with a plateau on \(\LiftStrip{i}\), while the additional buffer~\(\povp[2]\) is used only to localize the increment calculation. 

\begin{proof}[Proof of Lemma~\ref{dslp:lem:TestLocalizedPredictionEstimateNew}]
    Our goal is to derive an estimate for 
    \begin{equation*}
        \tnorm[\LiftStrip{i}]{\uLP{i}{n}-\uCNmod{n}},
    \end{equation*}
	where \(\uLP{i}{n}\) is the localized prediction, and \(\uCNmod{n}\) denotes a global CN step \eqref{dslp:eq:CNfirstOrderFormulation} with \(\uDS{n-1}\) as the previous approximation.
    The localized prediction \(\uLP{i}{n}\) is defined as 
    \begin{equation*}
        \uLP{i}{n} = \uDS{n-1} + \widehat{\delta}\widetilde u_{\chi_i}^n,
    \end{equation*}
    with \(\widehat{\delta}\widetilde u_{\chi_i}^n\) characterized by \eqref{dslp:eq:method:localizedPredictionIncrement}. For \(\uCNmod{n}\), we denote the corresponding increment by
    \begin{equation*}
        \delta \uCNmod{n} = \uCNmod{n} - \uDS{n-1}.
    \end{equation*}
    Since
	\begin{equation*}
		\NodalInt(\chi_i w_h)+\NodalInt((1-\chi_i)w_h)
		=
		\NodalInt(w_h)
		=
		w_h
		\qquad\text{for all } w_h\in\PkZeroTh,
    \end{equation*}
    we can decompose \(\delta \uCNmod{n}\) by superposition as 
    \begin{equation*}
		\delta \uCNmod{n}
		=
		\delta \widetilde{u}_{\chi_i}^n+\delta \widetilde{u}_{1-\chi_i}^n,
	\end{equation*}
    using \eqref{dslp:eq:TestLocalizedIncrement} with \(\Lambda=\chi_i\) and \(\Lambda=1-\chi_i\) to define \(\delta \widetilde{u}_{\chi_i}^n\) and \(\delta \widetilde{u}_{1-\chi_i}^n\), respectively.
    Therefore, since both \(\uCNmod{n}\) and \(\uLP{i}{n}\) start from \(\solDS{n-1}\), we have
	\begin{equation}
        \label{dslp:eq:localizedPredictionErrorDecomp}
		\uCNmod{n}-\uLP{i}{n}
		=
		\delta \widetilde{u}_{1-\chi_i}^n
		+
		\left(\delta \widetilde{u}_{\chi_i}^n-\widehat{\delta} \widetilde{u}_{\chi_i}^n\right).
	\end{equation}
    We estimate the two terms separately.  For the first term, the
	support of \(1-\chi_i\) is separated from~\(\LiftStrip{i}\) by~\(\povp[1]\) layers.
	Thus, Lemma~\ref{dslp:lem:TestLocalizedIncrementDecay}, applied with
	\(\Lambda=1-\chi_i\), gives
    \begin{equation}
        \label{dslp:eq:ApplicationTestLocalizedIncrementDecay}
		\tnorm[\LiftStrip{i}]{\delta \widetilde{u}_{1-\chi_i}^n}
		\leq
			C \Bigl(1 + \frac{\tau}{h}\Bigr)\gamma^{\povp[1]}
			\left(
				\norm[X_h(\fulldomain)][\big]{\solDS{n-1}}
				+\tau\norm[b,\fulldomain][\big]{\overline{f}_h^n}
			\right).
	\end{equation}
    For the truncation contribution, Lemma~\ref{dslp:lem:TestLocalizedIncrementTruncation}
	with \(\Lambda=\chi_i\) and \(\ell=\povp[2]\) gives
    \begin{equation}
        \label{dslp:eq:ApplicationTestLocalizedIncrementTruncation}
		\tnorm[\fulldomain][\big]{\delta \widetilde{u}_{\chi_i}^n-\widehat{\delta} \widetilde{u}_{\chi_i}^n}
		\leq
			\widetilde{C} \Bigl(1 + \frac{\tau}{h}\Bigr)\max \monosetc*{\frac{\tau}{h},1} \gamma^{\povp[2]}
			\left(
				\norm[X_h(\fulldomain)][\big]{\solDS{n-1}}
				+\tau\norm[b,\fulldomain][\big]{\overline{f}_h^n}
			\right).
	\end{equation}
    Thus, we set
    \begin{equation*}
        \CLoc \coloneqq \max \monosetc*{C \Bigl(1 + \frac{\tau}{h}\Bigr), \widetilde{C} \Bigl(1 + \frac{\tau}{h}\Bigr)\max \monosetc[\big]{\frac{\tau}{h},1}}. 
    \end{equation*} 
	Applying the triangle inequality to \eqref{dslp:eq:localizedPredictionErrorDecomp} together with the bounds \eqref{dslp:eq:ApplicationTestLocalizedIncrementDecay} and \eqref{dslp:eq:ApplicationTestLocalizedIncrementTruncation}
	yields the claim
	\begin{equation*}
		\tnorm[\LiftStrip{i}]{\uLP{i}{n}-\uCNmod{n}}
		\leq
				\CLoc
		\left(\gamma^{\povp[1]}+\gamma^{\povp[2]}\right)
		\left(
			\norm[X_h(\fulldomain)][\big]{\solDS{n-1}}
			+\tau\norm[b, \fulldomain]{\overline{f}_h^n}
		\right). \qedhere
	\end{equation*}
\end{proof}

\section{Exponential decay of local errors}
\label{dslp:sec:discreteSaintVenant} 
The goal of this section is to prove Lemma~\ref{dslp:lem:BoundLocalDiff} from
Section~\ref{dslp:sec:main-results}. The lemma bounds the local subdomain error
\begin{equation*}
	\norm[X_h(\SD{i})][\big]{\solDSloc{i}{n}  -  \solCNmod{n}}
\end{equation*}
by the prediction error on the lifting strip~\(\LiftStrip{i}\).  
Both the subdomain solution~\(\solDSloc{i}{n}\) and the comparison solution~\(\solCNmod{n}\) are obtained from one CN step starting from \(\solDS{n-1}\), the former locally on \(\SDov{i}\) and the latter globally on~\(\fulldomain\).

The argument consists of three steps. First, we identify the elliptic problem satisfied by the local difference between \(\uDSloc{i}{n}\) and \(\uCNmod{n}\) on \(\SDov{i}\).
Second, we apply the discrete Saint-Venant principle from \cite[Theorem~3.1]{BucD26} to bound the discrete energy on the non-overlapping subdomain \(\SD{i}\) by the discrete energy on the overlapping subdomain \(\SDov{i}\), multiplied by an exponentially decaying factor in the overlap width.
Third, using the well-posedness of the elliptic problem and Assumption~\ref{dslp:ass:boundaryLifting}, we bound the discrete energy of the local difference on \(\SDov{i}\) by the prediction error on the lifting strip.
We start with the first step and characterize the local difference on
\(\SDov{i}\) variationally.

\begin{lemma}[Local difference problem]
	\label{dslp:lem:LocalDifferenceOnSubdomains}
	Let \(i \in \monosetc*{1, \dots, \numberSD}\) and consider the subdomain \(\SDov{i}\). 
    The difference \( z_{u,i}^n \coloneq \uDSloc{i}{n} - \restr{\uCNmod{n}}{\SDov{i}}\) satisfies 
	\begin{equation}
		\label{dslp:eq:EllipticProblemLocalDifference}
		\begin{aligned}
			\bbil[\SDov{i}]{z_{u,i}^n, \varphi_i} + \tss[4] \abil[\SDov{i}]{ z_{u,i}^n,  \varphi_i} & = 0  &&\text{for all } \varphi_i \in \PkZeroTh[\SDov{i}], \\
			z_{u,i}^n &  %
			= \uLP{i}{n} - \uCNmod{n} \qquad && \text{on } \SDovint{i},\\
			z_{u,i}^n
			&=0 \qquad && \text{on } \dSDov{i}\setminus\SDovint{i}.
		\end{aligned}
	\end{equation}
	Moreover, for \( z_{v,i}^n = \vDSloc{i}{n} - \restr{\vCNmod{n}}{\SDov{i}} \), we have
	\begin{equation}
		\label{dslp:eq:EllipticProblemLocalDifference_p}
		z_{u,i}^n = \ts[2] z_{v,i}^n \quad \text{ in } \PkTh[\SDov{i}].
	\end{equation}
\end{lemma}

\begin{proof}
	As in \cite[Lemma 6.1]{BucH25CG}, the proof follows by straightforward calculations.
\end{proof}

Let \(D_h \coloneq \Th[\SD{i}]\) and recall that \(\Th[\SDov{i}] = \Patch{\ovp}{D_h}\) by construction.
As in \cite[Section~4.2]{BucD26} and similar to \cite[Definition 3.3]{MalP14}, we define a cutoff function \(\cutoff{\ell} \in \PkTh\) for each \(\Patch{\ell}{D_h}\) with 
\begin{subequations}
    \label{dslp:eq:AssCutoff}
    \begin{alignat}{2}
         \label{dslp:eq:AssCutoff:A} 0 \leq \cutoff{\ell} &\leq 1, &&\text{on } \fulldomain, \\
        \label{dslp:eq:AssCutoff:B} \cutoff{\ell} &= 0, &&\text{on } K \in \Th \setminus \Patch{\ell+1}{D_h}, \\
        \label{dslp:eq:AssCutoff:C} \cutoff{\ell} &= 1, &&\text{on } K \in \Patch{\ell}{D_h}, \\
        \label{dslp:eq:AssCutoff:D} \norm[L^{\infty}(\fulldomain)][\big]{\grad \cutoff{\ell}} &\leq \CCutoff h^{-1} \; ,
    \end{alignat}
\end{subequations}
with a constant \(\CCutoff>1\) independent of \(h\). Using this construction we show that the discrete elliptic energy with respect to the variational problem \eqref{dslp:eq:EllipticProblemLocalDifference} contained in successive element patches around the initial cell set \(D_h = \Th[\SD{i}]\) contracts with a factor strictly smaller than one. This implies an exponential decay of the discrete energy with respect to the number of mesh layers, in our case the overlap parameter \(\ovp\). 

\begin{lemma}[Discrete Saint-Venant principle]
	\label{dslp:lem:DSV-decay}
	Let \(z_{u,i}^n \in \PkTh[\SDov{i}]\) be the discrete solution to the variational formulation~\eqref{dslp:eq:EllipticProblemLocalDifference}.
	Then, 
	\begin{equation}
		\label{dslp:eq:dSV:VolumeBound}
		\tnorm[\SD{i}][\big]{z_{u,i}^n} \leq \rho^{\ovp} \tnorm[\SDov{i}][\big]{z_{u,i}^n}, \qquad \text{with} \qquad \rho = \Bigl(\frac{M_{\tau/h}}{1 + M_{\tau/h}} \Bigr)^{\half} < 1,
	\end{equation}
	and 
	\begin{equation*}
		M_{\tau/h} = \Cint \max \monosetc*{1, \frac{\CCutoff}{2}} \frac{\beta^{\half}}{\alpha^{\half}} \bigl( 1 + \frac{\tau}{2 h}\bigr).
	\end{equation*}
\end{lemma}
\begin{proof}
	Compared with \cite[Theorem~3.1 a)]{BucD26} with \(\lambda = 2/\ts\), the only additional point is the presence of the material coefficient \(B\) in the weighted \(L^2\)-term.
	The proof carries over after replacing the unweighted \(L^2\)-bounds by the corresponding \(B\)-weighted bounds.
	More precisely, one only has to adapt the integral bounds in \cite[Theorem~4.2 c) and d)]{BucD26}.
	This changes only the constants, through the material bounds \(\alpha\) and \(\beta\), and leads to the stated definition of \(M_{\tau/h}\).
\end{proof}

With Lemma~\ref{dslp:lem:DSV-decay} at hand, we can prove the main result of this section.

\begin{proof}[Proof of Lemma~\ref{dslp:lem:BoundLocalDiff}]
    By the well-posedness of the variational problem \eqref{dslp:eq:EllipticProblemLocalDifference}, we estimate the~\mbox{\(\tnorm[\SDov{i}]{\cdot}\)}-norm of the solution \(z_{u,i}^n\) against the respective norm of the boundary lifting 
    \begin{equation}
        \label{dslp:eq:aPrioriEstimateLifting}
        \tnorm[\SDov{i}]{z_{u,i}^n} \leq 2
		\tnorm[\SDov{i}][\big]{\blifting{i}\bigl(\restr{\uLP{i}{n}-\uCNmod{n}}{\SDovint{i}}\bigr)}.
    \end{equation}

	Combining \eqref{dslp:eq:aPrioriEstimateLifting} with Assumption~\ref{dslp:ass:boundaryLifting} gives
	\begin{equation}
		\label{dslp:eq:aPrioriEstimateLifting-2}
		\tnorm[\SDov{i}]{z_{u,i}^n}
		\leq
		2\Cblift h^{-\half}
		\tnorm[\LiftStrip{i}]{\uLP{i}{n}-\uCNmod{n}}.
	\end{equation}
	The discrete Saint-Venant principle from Lemma~\ref{dslp:lem:DSV-decay} together with \eqref{dslp:eq:aPrioriEstimateLifting-2} yields
	\begin{equation*}
		\tnorm[\SD{i}]{z_{u,i}^n} 
		\leq \rho^{\ovp}\tnorm[\SDov{i}]{z_{u,i}^n}
		\leq 2\Cblift h^{-\half} \rho^{\ovp}  
		\tnorm[\LiftStrip{i}]{\uLP{i}{n}-\uCNmod{n}},
	\end{equation*}
	which is the estimate stated in Lemma~\ref{dslp:lem:BoundLocalDiff}.
\end{proof}

\section{Averaging in the discrete energy space}
\label{dslp:sec:averaging}
In this last theoretical section, we prove Lemma~\ref{dslp:Lem:AvgStabilityXh}, which is a stability result for the averaging operator \(\averaging\) from \eqref{dslp:eq:method:averaging} in the discrete energy space \(X_h\).
The averaging operator itself is the same as in \cite[Section~3.3]{BucH25CG}. On the interface nodes of the non-overlapping subdomains it coincides with the standard nodal averaging operator \(J_h^\mathrm{av}\) from \cite[Section~22.2]{ErnG21I}.
As a preparatory step, we first derive a weighted \(L^2\)-stability estimate in the \(\norm[b]{\cdot}\)-norm, which is used to control the \(v\)-component in the subsequent \(X_h\)-stability argument.

\begin{lemma}[Weighted \(L^2\)-stability of averaging]
	\label{dslp:lem:AvgL2Stability}
		The averaging operator \( \averaging \) from \eqref{dslp:eq:method:averaging} satisfies a stability estimate in the weighted~\(L^2\)-norm for~\(\vComponent_i \in \PkTh[\SDov{i}], \; i\in\monosetc{1,\dots,\numberSD}\), in the sense of
		\begin{equation}
			\label{dslp:eq:averagingBnorm}
			\norm[b,\fulldomain][\big]{\averaging \big( \big\{ \vComponent_i \big\}_{i=1,\dots,\numberSD}\big)}^2 \leq \cavg \sumSD \norm[b, \SD{i}][\big]{\vComponent_i}^2
		\end{equation}
		with a constant \( \cavg > 0 \) which is independent of \(h\) and \(\tau\).
\end{lemma}

\begin{proof}
	The averaging operator~\(\averaging\) from \eqref{dslp:eq:method:averaging} coincides by definition with \(J_h^\mathrm{av}\) from \cite[Section~22.2]{ErnG21I}.
	The proof follows as in \cite[Lemma~7.1]{BucH25CG}, see also \cite[Corollary~22.4]{ErnG21I}.
\end{proof}

Lemma~\ref{dslp:lem:AvgL2Stability} now allows us to prove the stability of \(\averaging\) in the \(\norm[X_h]{\cdot}\)-norm. 
The key idea of this proof is to change from the \(u\)-component into the \(v\)-component, use an inverse estimate to bring everything into the \(\norm[b]{\cdot}\)-norm and then apply Lemma~\ref{dslp:lem:AvgL2Stability}.

\begin{proof}[Proof of Lemma~\ref{dslp:Lem:AvgStabilityXh}]
	Our goal is to convert the local subdomain error bounds from the previous sections into a global \(\norm[X_h]{\cdot}\)-estimate after averaging. 
	More precisely, we derive the bound \eqref{dslp:eq:avg_prop3} given by 
	\begin{equation*}
		\norm[X_h(\fulldomain)][\big]{\solDS{n}- \solCNmod{n}}^2 \leq \cavg \Cinv \frac{\beta}{\alpha}\Bigl(  \tss[h^2] +1 \Bigr) \sumSD \norm[X_h(\SD{i})][\big]{\solDSloc{i}{n}- \solCNmod{n}}^2,
	\end{equation*}
	with \(\cavg\) and \(\Cinv\) defined in Lemma~\ref{dslp:lem:AvgL2Stability} and Lemma~\ref{dslp:Lem:InverseEstimate}, respectively.
	For that, let \(i\in\monosetc{1,\dots,\numberSD}\) and set
	\begin{equation}
		z_i \coloneqq \solDSloc{i}{n} - \restr{\solCNmod{n}}{\SDov{i}}
		= \pmat{z_{\uComponent,i}\\ z_{\vComponent,i}} .
	\end{equation}
	Recall that a globally continuous function is preserved when restricting it to the overlapping subdomains and applying \(\averaging\) on these restrictions. 
	By \eqref{dslp:eq:method:globalDSupdate} and componentwise linearity of~\(\averaging\) we thus have
	\[
		\solDS{n}-\solCNmod{n}
		=
		\averaging\Big( \big\{ z_i \big\}_{i=1,\dots,\numberSD} \Big).
	\]
	Moreover, \eqref{dslp:eq:EllipticProblemLocalDifference_p} in Lemma~\ref{dslp:lem:LocalDifferenceOnSubdomains} gives
	\(z_{\uComponent,i}=\ts[2]z_{\vComponent,i}\) on \(\SDov{i}\). Hence, again by
	linearity of \(\averaging\),
	\[
		\uDS{n}-\uCNmod{n}
		=
		\ts[2]\bigl(\vDS{n}-\vCNmod{n}\bigr).
	\]
	Using the definition of the \(X_h(\fulldomain)\)-norm and the inverse inequality (Lemma~\ref{dslp:Lem:InverseEstimate}), we therefore obtain
	\begin{align*}
		\norm[X_h(\fulldomain)][\big]{\solDS{n}- \solCNmod{n}}^2
			& =
		\norm[a,\fulldomain]{\uDS{n}-\uCNmod{n}}^2
		+\norm[b, \fulldomain]{\vDS{n}-\vCNmod{n}}^2
		\\
			& \leq
		\bigl(\Cinv \frac{\beta}{\alpha} \tss[h^2]+1\bigr)
		\norm[b, \fulldomain]{\vDS{n}-\vCNmod{n}}^2 .
	\end{align*}
	Applying Lemma~\ref{dslp:lem:AvgL2Stability} to the scalar averaged
	velocity component yields
	\begin{align*}
		\norm[b, \fulldomain]{\vDS{n}-\vCNmod{n}}^2
			=
		\norm[b, \fulldomain][\big]{
			\averaging\Big( \big\{ z_{\vComponent,i}\big\}_{i=1,\dots,\numberSD}\Big)
		}^2 \leq
		\cavg\sumSD \norm[b, \SD{i}]{z_{\vComponent,i}}^2 \leq
		\cavg\sumSD \norm[X_h(\SD{i})][\big]{z_i}^2 .
	\end{align*}
	Combining the last two estimates proves \eqref{dslp:eq:avg_prop3}.
\end{proof}

\section{Discussion on the overlap condition}
\label{dslp:sec:discussion}
In the main result of this work, Theorem~\ref{dslp:thm:ConvergenceOverlapCondition}, we impose a condition on the overlap parameters
\(\ovp,\povp[1]\), and \(\povp[2]\). More precisely, the error bound
\eqref{dslp:eq:AbstractConvergenceTheoremBound} depends on
\[
    \vartheta
    =
    \mathfrak{C}_{\tau,h} \numberSD^{\half}
    \rho^{\ovp}
    \bigl(\gamma^{\povp[1]}+\gamma^{\povp[2]}\bigr),
\]
where \(\numberSD\) is the number of subdomains and the factor \(\mathfrak{C}_{\tau,h} \) is given in \eqref{dslp:eq:factor-mathfrakC} by 
\begin{equation*}
    \mathfrak{C}_{\tau,h} = (\cavg \Cinv)^{\half} \frac{\beta^{\half}}{\alpha^{\half}}\Bigl(  \ts[h] +1 \Bigr) 2 \Cblift h^{-\half} \CLoc,
\end{equation*}
where \(\CLoc\) grows at most quadratically in the ratio \(\tau/h\), but depends neither on \(\tau\) nor on \(h\) individually.
The factor \(\mathfrak{C}_{\tau,h}\) is thus independent of the overlap parameters \(\ovp,\povp[1],\povp[2]\).
The polynomial convergence rate in \eqref{dslp:eq:ConvergenceTheoremBound}
is obtained once \(\vartheta\leq \tau^{p+1}\) for some \(p>0\).

In this section, we discuss how this condition should be understood. The main
point is that it is not an artificial time step restriction, but a natural
resolution requirement on the overlap widths. The prediction overlaps
\(\povp[1]\) and \(\povp[2]\) have to be large enough to capture the physical information
that can reach the artificial interfaces during one time step, while the
subdomain overlap \(\ovp\) provides the required decay to damp the remaining
interface error. This interpretation also motivates the practical parameter
choices used in the numerical experiments in
Section~\ref{dslp:sec:numericalExperiments}.

By choosing \(\povp[1]\geq\povp[2]\) the expression for \(\vartheta\) can be bounded by
\begin{equation*}
    \vartheta 
    \leq
    2
    \mathfrak{C}_{\tau,h}  \numberSD^{1/2}
    \rho^{\ovp}
    \gamma^{\povp[2]} .
\end{equation*}
Here, \(\gamma < 1\) is the localization factor from Lemma~\ref{dslp:lem:TestLocalizedPredictionEstimateNew} and \(\rho <1\) is the decay factor from Lemma~\ref{dslp:lem:BoundLocalDiff}.

The error stays small on a fixed time interval if \(\numberTS\vartheta\ll 1\), i.e., it is sufficient to require
\begin{equation*}
    \rho^{\ovp}\gamma^{\povp[2]}
    \in \mathcal{O}(\frac{\tau^{p+1}}{2\mathfrak{C}_{\tau,h} \numberSD^{\half}}).
\end{equation*}
For \(\alpha\) and \(\beta\) fixed, if \(\tau/h\geq1\), then 
\begin{equation*}
    \gamma^2\simeq (\tau / h)/(1+ \tau / h), 
\end{equation*}
and thus
\(\abs{\log\gamma}\simeq h/\tau\).  Hence, for moderate \(\ovp\), the prediction
overlap has to satisfy
\begin{equation}
    \label{dslp:heuristic:povp2}
        \povp[2]\geq C \frac{\tau}{h}
    \log\left(
            \frac{\beta}{\alpha}\Bigl(\frac{\tau}{h} + 1\Bigr)\frac{\numberSD^{\half}}{\tau h^{\half}}
    \right),
\end{equation}
with a constant \(C>0\) that depends on \(p\) and the additional decay gained from \(\rho^{\ovp}\). 
This means that the additional overhead from the factor \(\mathfrak{C}_{\tau,h}\) only scales logarithmically in \(\numberSD\) and \(h\).
In the numerical experiments reported in the next section, we use an even
simpler practical choice. We do not include the logarithmic factor suggested by
\eqref{dslp:heuristic:povp2}. Instead, for large \(\tau/h\), we choose
\[
    \povp[1]=\povp[2]\simeq \tau/h .
\]
In essence, we bound the logarithmic overhead by a fixed constant. Moreover, when \(\tau\) is comparable to \(h\), we only enforce a small minimal prediction
overlap so that the construction remains meaningful; in the experiments below,
we typically use \(\povp[1]=2\) and \(\povp[2]=1\) in this regime. These choices
are sufficient in our tests to make the domain splitting scheme with localized
prediction well-defined and accurate enough to compete with global implicit methods.
At the same time, the method retains the structural advantages of local
subdomain solves and is therefore well suited for parallelization. The next
section illustrates these observations numerically.

\section{Numerical experiments}
\label{dslp:sec:numericalExperiments}

In this section, we illustrate the behavior of the proposed domain splitting
method with localized implicit prediction (DSLP). The implementation is based
on the FEniCSx framework \cite{dolfinx,dolfinx2}. We compare DSLP with the
domain splitting method from \cite{BucH25CG}, study the influence of the overlap
parameters \(\ovp\), \(\povp[1]\), and \(\povp[2]\), and test the method in
one and two spatial dimensions. We close with a short discussion of the
parallel implementation and first work-precision experiments.

The code corresponding to the experiments in this section is made publicly
available at
\begin{equation*}
    \text{\url{https://github.com/tim-buchholz/dslp-acoustic-wave.git}} \; .
\end{equation*}

Throughout the section, CN denotes the global Crank--Nicolson method, LF the
leapfrog method, and DS the domain splitting scheme from \cite{BucH25CG}. Unless
specified otherwise, the prediction overlaps are chosen according to
\begin{equation}
    \label{dslp:eq:numerics-prediction-overlaps}
    \povp[1]
    =
    \max \monosetc*{m_1, \lceil \eta \kappa \ts h^{-1}\rceil},
    \qquad
    \povp[2]
    =
    \max \monosetc*{m_2, \lceil \eta \kappa \ts h^{-1}\rceil},
\end{equation}
where \(\kappa\) is the maximal characteristic wave speed induced by the material parameters. 
In the homogeneous examples below, \(A\equiv B\equiv 1\), and therefore
\(\kappa\equiv(B^{-1}A)^{1/2}\equiv1\). 
The integer pair \((m_1,m_2)\) gives a minimal prediction width. 
In the experiments, we generally use \((m_1,m_2)=(2,1)\), which keeps
the construction of the prediction domains and cutoff functions well-defined
also for small values of \(\tau h^{-1}\). The parameter \(\eta\) scales the
prediction width in the regime \(\tau \geq c h\) for some \(c>0\); it plays the role of the
logarithmic safety factor suggested by the analysis, but in the experiments
already the choice \(\eta=1\) is often sufficient.

\subsection{Traveling pulse example from \texorpdfstring{\cite{BucH25CG}}{BucH25CG}}

We first revisit the one-dimensional traveling pulse example from
\cite[Section~9.1]{BucH25CG}, originally inspired by \cite{Lin22}. On
\(\Omega=(0,1)\), we consider the homogeneous wave equation with
\(A\equiv B\equiv1\). For \(\xi\in(0,1)\) and \(s>0\), let
\begin{equation}
    \label{dslp:eq:numerics-pulse}
    \mu_{\xi,s}(z)
    =
    \mathbbm{1}_{\{\abs{z-\xi}<s\}}
    \sin\!\left(\mfrac{z-(\xi+s)}{2s}\pi\right)^3 ,
\end{equation}
where \(\mathbbm{1}_S\) denotes the indicator function of the set \(S\).
The initial displacement is chosen as
\begin{equation}
    \label{dslp:eq:numerics-1d-initial-data}
    u^0(x)=\mu(x),
    \qquad
    \mu(z)=\mu_{\xi_1,s}(z)-\mu_{\xi_2,s}(z),
    \qquad
    \xi_1=0.55,\quad \xi_2=0.45,\quad s=0.2,
\end{equation}
and the initial velocity is \(v^0(x)=\restr{\partial_t\mu(x-t)}{t=0}\).
The error is measured at the final time \(T=1\) as a relative
\(H_0^1\times L^2\) error.

This example is useful as a direct comparison with the earlier DS method. Since
that method uses a leapfrog prediction with mass lumping, the direct comparison
is restricted to first-order finite elements with mass lumping. For both
domain splitting schemes we choose the overlap \(\ovp=8\). For DSLP, we use
\eqref{dslp:eq:numerics-prediction-overlaps} with \(\eta=1\) and
\((m_1,m_2)=(2,1)\).

Figure~\ref{dslp:Fig:1DcomparisonFEMml} shows the error with respect to the
exact solution and the difference from the global CN approximation,
both measured in the same relative \(H_0^1\times L^2\) norm. 
The original DS method improves the stability range compared with the fully explicit leapfrog method, but it still exhibits a CFL-type restriction. 
In contrast, the DSLP method remains stable for all tested time steps. 
Its error is essentially indistinguishable from the CN error on the log-log scale. 
The same behavior is visible in the differences from CN: the DS--CN difference is stable only below a CFL-type threshold, whereas the DSLP--CN difference remains bounded throughout the tested range and is smaller than the error of CN with respect to the exact solution.

\begin{figure}[t]
        \centering
        \includegraphics{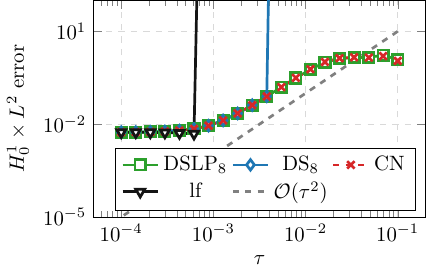}
        \includegraphics{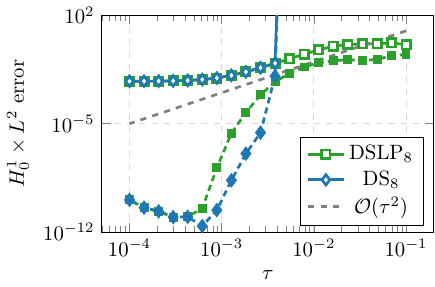}
    \caption{Traveling pulse example from \cite[Section~9.1]{BucH25CG}. First-order finite elements with mass lumping, \(\ovp=8\), \(\eta=1\), and minimal prediction width \((m_1,m_2)=(2,1)\). Left: errors with respect to the exact solution. Right: differences from the global CN approximation depicted by the dashed lines with solid markers.}
    \label{dslp:Fig:1DcomparisonFEMml}
\end{figure}

The localized implicit prediction is not tied to mass lumping or to first-order
finite elements. We therefore repeat the same experiment with standard
conforming finite elements of degrees \(k=1,2,3,4\), again using
\(\ovp=8\), \(\eta=1\), and \((m_1,m_2)=(2,1)\). The results are shown in
Figure~\ref{dslp:Fig:FemDegrees1D}. The DSLP method is robust with respect to
the polynomial degree. No stability deterioration is observed when the degree is
increased. This is in clear contrast to explicit schemes, whose admissible time
step becomes even more restrictive for higher-order finite elements. Moreover,
the DSLP--CN difference tends to decrease for higher polynomial degrees in this
experiment.

\begin{figure}[t]
    \begin{center}
        \includegraphics{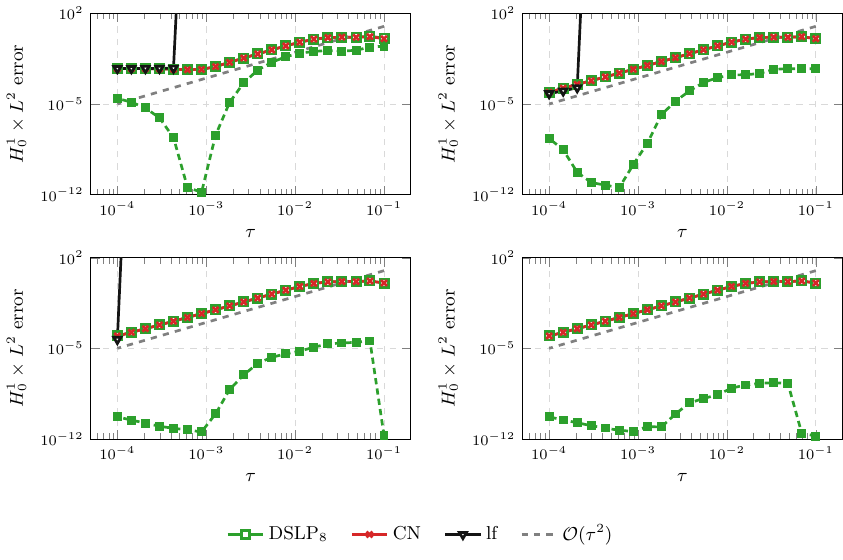}
    \end{center} 
    \caption{Traveling pulse example with standard conforming finite elements of degrees \(k=1,2,3,4\), \(\ovp=8\), \(\eta=1\), and minimal prediction width \((m_1,m_2)=(2,1)\). Dashed lines with solid markers show the difference from the global CN approximation.}
    \label{dslp:Fig:FemDegrees1D}
\end{figure}
 
\subsection{Influence of the overlap parameters}

Next, we investigate how the parameters \(\ovp\), \(\povp[1]\), and~\(\povp[2]\) affect the difference between DSLP and CN. The experiment again
uses the one-dimensional traveling pulse example, now with standard finite
elements of degree \(k=2\) and mesh size \(h=10^{-3}\). Since the previous
figures already show that the total DSLP error is dominated by the global CN
discretization error for the parameter choices used there, we focus here only
on the difference between DSLP and CN. Figure~\ref{dslp:Fig:DifferenceCN-ParameterDependence} reports on three different experiments each investigating the influence of different parameters.

The parameter \(\ovp\) controls the width of the overlap used for the subdomain
CN solves. Increasing \(\ovp\) produces a nearly uniform reduction
of the DSLP--CN difference over all tested time steps. This agrees with the
factor \(\rho^{\ovp}\) in Theorem~\ref{dslp:thm:ConvergenceOverlapCondition}.

The parameter \(\eta\) scales the prediction widths in
\eqref{dslp:eq:numerics-prediction-overlaps}. Its effect is most visible for
\(\tau>h\), where the term \(\lceil\eta\kappa\tau h^{-1}\rceil\) determines
\(\povp[1]\) and \(\povp[2]\). Larger values of \(\eta\) reduce the error made
in the prediction stage and therefore lower the DSLP--CN difference in this
regime. For the present example, \(\eta=1\) already gives a satisfactory
accuracy, but increasing \(\eta\) provides a simple way to further reduce the
prediction error.

Finally, the minimum values \((m_1,m_2)\) primarily influence the regime
\(\tau<h\), where the scaled term
\(\lceil\eta\kappa\tau h^{-1}\rceil\) is small. The baseline choice
\((m_1,m_2)=(2,1)\) is sufficient for the geometric requirements of the
prediction construction, in particular for the cutoff functions used near the
interfaces. Larger minimum widths reduce the DSLP--CN difference for small
time steps. This can be useful in examples where the CN error itself is very
small and one wants the splitting error to remain even smaller.

Overall, the three parameters affect complementary parts of the error. The
subdomain overlap~\(\ovp\) reduces the splitting error globally, the scaling
\(\eta\) controls the prediction accuracy for \(\tau\geq c h\) for some \(c>0\), and the
minimum prediction widths determine the small-\(\tau\) baseline. This makes it
possible to tune the method so that the DSLP--CN difference is below the error
of the global CN method with respect to the exact solution.

\begin{figure}[t]
    \begin{center}
        \includegraphics{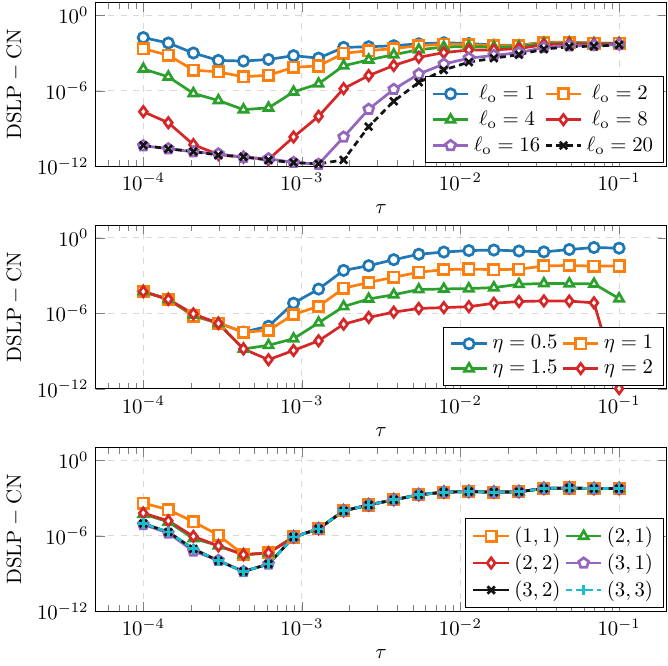}
    \end{center}
    \caption{Parameter study for the traveling pulse example. Shown are differences between DSLP and the global CN approximation for standard finite elements of degree \(k=2\) and \(h=10^{-3}\). Top: sweep over \(\ovp\) with \(\eta=1\) and \((m_1,m_2)=(2,1)\). Middle: sweep over \(\eta\) with \(\ovp=4\) and \((m_1,m_2)=(2,1)\). Bottom: sweep over the minimum prediction width \((m_1,m_2)\) with \(\ovp=4\) and \(\eta=1\).}
    \label{dslp:Fig:DifferenceCN-ParameterDependence}
\end{figure}

\subsection{Two-dimensional example}

We now consider the two-dimensional extension of the traveling pulse example
from \cite[Section~9.2]{BucH25CG}. Let \(\Omega=(0,1)^2\), \(\xi=0.5\), and
\(s=0.2\). With the one-dimensional pulse
\(\mu_{\xi,s}\) from \eqref{dslp:eq:numerics-pulse}, we prescribe the exact
displacement by
\begin{equation*}
    u(x,y,t)
    =
    u_{\mathrm{1D}}(x,t)\mu_{\xi,s}(y)
    +
    u_{\mathrm{1D}}(y,t)\mu_{\xi,s}(x),
\end{equation*}
where \(u_{\mathrm{1D}}\) solves the homogeneous one-dimensional wave equation
on \((0,1)\) with initial displacement \(\mu_{\xi,s}\). The right-hand side
\(f\) is obtained by inserting this expression into the wave operator.

The computational domain is discretized by a triangular mesh with \(h=0.005\).
We use standard conforming finite elements of degree \(k=2\), choose
\(\eta=1\) and \((m_1,m_2)=(2,1)\), and compare the overlap choices
\(\ovp=4\) and \(\ovp=8\). 
The prediction overlap parameters \(\povp[1]\) and \(\povp[2]\) are again chosen according to \eqref{dslp:eq:numerics-prediction-overlaps}. The domain is decomposed into \(16\) subdomains in a
checkerboard configuration, so the decomposition includes cross-points. As in
the one-dimensional case, the error is measured at the final time \(T=1\) as a
relative \(H_0^1\times L^2\) error.

The results in Figure~\ref{dslp:Fig:2Dconvergence} are consistent with the one-dimensional observations.
The leapfrog method suffers from a strict CFL condition. For the tested mesh, it is stable only for time-step sizes at which the error is already dominated by the spatial discretization.
The errors of DSLP with \(\ovp=4\) and
\(\ovp=8\) agree closely with the error of the global CN method,
while no stability restriction is observed for the tested time steps. 
The difference between DSLP and CN remains smaller than the CN error with respect to the exact solution. 
Increasing the overlap from~\(\ovp=4\) to~\(\ovp=8\) further reduces this difference, in agreement with the exponential overlap
damping predicted by the analysis.

These results indicate that the parameter choices established in the
one-dimensional experiments transfer robustly to the two-dimensional setting.
In particular, the method retains the stability behavior of the global implicit
scheme while using only localized implicit predictions and independent
subdomain solves.

\begin{figure}[t]
    \begin{center}
        \includegraphics{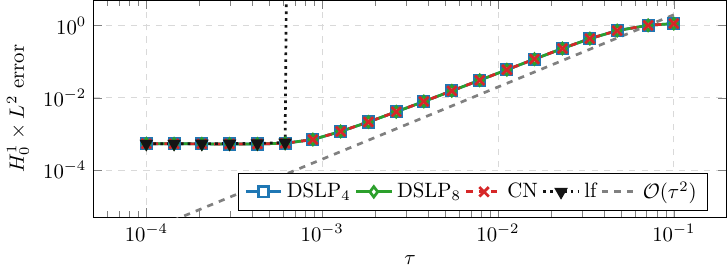}
        \includegraphics{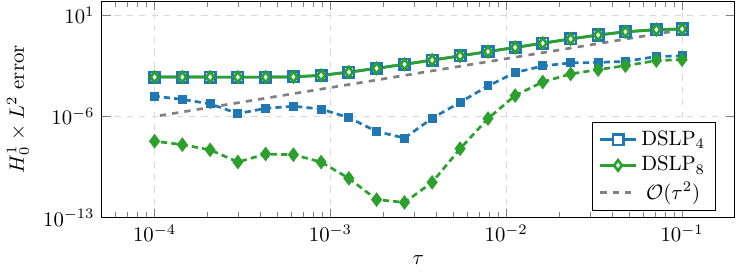}
    \end{center}
     \caption{Two-dimensional traveling pulse example from \cite[Section~9.2]{BucH25CG}. Unit square, triangular mesh with \(h=0.005\), standard finite elements of degree \(k=2\), \(\eta=1\), and \((m_1,m_2)=(2,1)\). Top: errors with respect to the exact solution for \(\ovp=4\) and \(\ovp=8\). Bottom: dashed lines with solid markers show the differences from the global CN approximation.}
     \label{dslp:Fig:2Dconvergence}
\end{figure}

\subsection{Parallel implementation}

We also implemented the method with distributed-memory parallelism. 
For this implementation, meshes are generated with Gmsh \cite{gmsh} and partitioned with PT-Scotch \cite{ptscotch}. 
The local finite element problems are assembled and solved
with FEniCSx \cite{dolfinx,dolfinx2} and PETSc through petsc4py
\cite{petsc4py}. Communication between subdomains is organized explicitly by
point-to-point MPI communication \cite{mpi}. This communication layer exchanges
the data needed on neighboring overlaps and prediction patches.

The execution is split into two phases. In the first phase, the mesh is
generated, partitioned, and the communication pattern between subdomains is set
up. In the second phase, the local finite element spaces, solvers, and data
structures are initialized, followed by the time loop, error evaluation, and
optional plotting. The global CN reference uses the parallel
finite element and linear algebra infrastructure provided by FEniCSx and PETSc.

For the parallel work-precision test, we consider the traveling pulse example
from the previous section; see also \cite[Section~9.2]{BucH25CG}. The global CN
and DSLP computations use a mesh size \(h=2\cdot 10^{-3}\) and finite elements
of degree \(k=2\) (without mass lumping). 
Errors are measured relative to a reference solution computed with CN on a finer mesh with \(h=10^{-3}\), degree \(k=2\), and time step \(\tau=2.5\cdot 10^{-5}\).
All parallel runs use \(64\) MPI ranks, corresponding to \(64\) subdomains.
For DSLP we set \(\ovp=4\), \(\eta=1\), and use the minimal
prediction overlaps \((m_1,m_2)=(2,1)\). The actual prediction overlaps are
chosen separately for each time step according to the heuristic described
in \eqref{dslp:eq:numerics-prediction-overlaps}. The time step is varied logarithmically between \(10^{-2}\) and
\(10^{-4}\), with minor adjustments so that the final time \(T=1\) is reached
exactly. The computations were run on a bwUniCluster node of type
\texttt{cpu\_il}, using one Ice Lake node with \(64\) MPI ranks.

Figure~\ref{dslp:Fig:ParallelWorkPrecision} compares the work-precision
behavior of DSLP with the parallelized implementation of the global CN method.
The left plot reports only the measured wall time spent in the time loop, while the right plot includes the external wall time and therefore also reflects the different setup and meshing costs. 
In the time-loop comparison, DSLP reaches essentially the same accuracy as CN for sufficiently small time steps, but with a reduced wall time per step. 
The speedup compared to the parallelized CN is about \(3\) as visualized by the dashed gray line. 

 \begin{figure}[t]
    \begin{center}
        \includegraphics{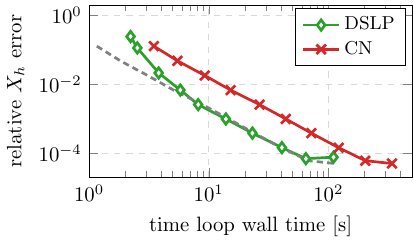}
        \includegraphics{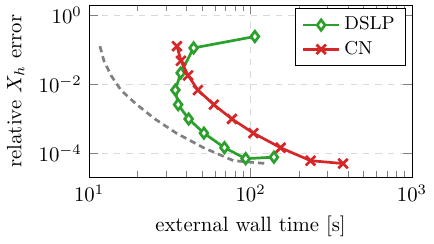}
    \end{center}
     \caption{Work-precision diagram for the two-dimensional traveling pulse example. The plot compares DSLP with the parallel implementation of the global CN for time steps between \(10^{-2}\) and \(10^{-4}\). The dashed line corresponds to a speedup of \(3\) compared to the parallel implementation of the CN scheme.}
     \label{dslp:Fig:ParallelWorkPrecision}
\end{figure}

 \begin{figure}[t]
    \begin{center}
        \includegraphics{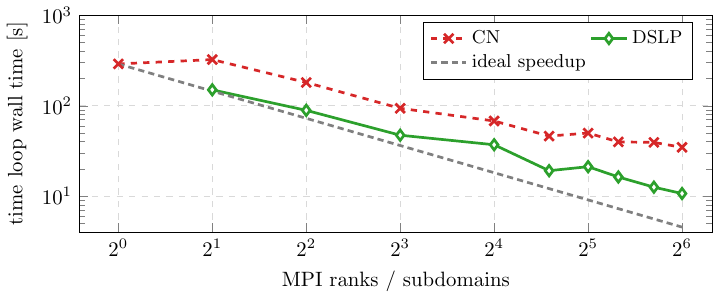}
    \end{center}
     \caption{Strong scaling for the traveling pulse example with fixed
    \(h=2\cdot 10^{-3}\), \(k=2\), \(\tau=10^{-3}\), and \(T=1\). We compare the
    wall time spent in the time loop for the parallel CN method and the DSLP method.
    For DSLP, the number of MPI ranks equals the number of non-overlapping
    subdomains, and we use \(\ovp=4\), \(\eta=1\), and minimal
    prediction overlaps \((m_1,m_2)=(2,1)\). The dashed gray line indicates ideal strong scaling, normalized by the serial CN runtime.}
     \label{dslp:Fig:StrongScaling}
\end{figure}

The structural benefits for parallelization compared to the parallelized implementation of the global CN scheme become even clearer in a strong scaling experiment.
In Figure~\ref{dslp:Fig:StrongScaling} we plot the wall time spent in the actual time loop against the number of MPI ranks used for the simulation. 
In the case of the DSLP method, the number of MPI ranks coincides with the number of subdomains.
At 64 MPI ranks, the DSLP time loop takes \(10.7\) seconds, compared with \(34.6\) seconds for CN, corresponding to a relative speedup of \(3.2\). 

In summary, based on the presented results we can conclude that the proposed domain splitting method with localized implicit prediction attains stability and accuracy comparable to global implicit time integration while providing structural benefits for parallelization.

\appendix

\section*{Acknowledgments}  
We gratefully acknowledge funding by the Deutsche Forschungsgemeinschaft (DFG, German Research Foundation) -- Project-ID 258734477 -- SFB 1173.
The authors acknowledge support by the state of Baden-Württemberg through bwHPC.

\bibliographystyle{amsplain}	
\bibliography{bibliography}

\end{document}